\documentclass{amsart}
\usepackage{amsmath}
\usepackage{amsthm}
\usepackage{amsfonts,mathrsfs}
\usepackage{amssymb}
\usepackage{graphicx}
\usepackage{enumitem}
\usepackage{color}
\usepackage{verbatim}
\theoremstyle{plain}
\newtheorem{theorem}{Theorem}[section]
\newtheorem{lemma}[theorem]{Lemma}
\newtheorem{example}[theorem]{Example}

\newtheorem{proposition}[theorem]{Proposition}

\theoremstyle{definition}

\newtheorem{problem}[theorem]{Problem}
\newtheorem{result}[theorem]{Result}

\newtheoremstyle{TheoremNum}
	{\topsep}{\topsep}              
  {\itshape}                      
  {}                              
  {\bfseries}                     
  {.}                             
  { }                             
  {\thmname{#1}\thmnote{ \bfseries #3}}
\newtheorem{remark}[theorem]{Remark}

\newcommand{\Z}{\mathbb Z}

\newcommand{\cC}{\mathcal C}

\newcommand{\RN}[1]{%
  \textup{\uppercase\expandafter{\romannumeral#1}}%
}

 \def\zhou#1 {\fbox {\footnote {\ }}\ \footnotetext { From Yue: {\color{red}#1}}}

 \def\chen#1 {\fbox {\footnote {\ }}\ \footnotetext { From Chen: {\color{blue}#1}}}

\begin{document}
	\title{On almost perfect linear Lee codes of packing radius $2$}
	\author[X. Xu]{Xiaodong Xu\textsuperscript{\,1}}
	\address{\textsuperscript{1}Guangxi Academy of Sciences, Nanning, 530007, P.R.China.}
	\email{xxdmaths@sina.com}
	\author[Y. Zhou]{Yue Zhou\textsuperscript{\,2, $\dagger$}}
	\address{\textsuperscript{2}Department of Mathematics, National University of Defense Technology, 410073 Changsha, China}
	\address{\textsuperscript{$\dagger$}Corresponding author}
	\email{yue.zhou.ovgu@gmail.com}
	\thanks{The extended abstract \cite{xu_wcc_2022} of an earlier version of this paper was presented in the 12th International Workshop on Coding and Cryptography (WCC) 2022.}
	\date{\today}
	\begin{abstract}
		More than 50 years ago, Golomb and Welch conjectured that there is no perfect Lee codes $C$ of packing radius $r$ in $\mathbb{Z}^{n}$ for  $r\geq2$ and $n\geq 3$. Recently, Leung and the second author proved that if $C$ is linear, then the Golomb-Welch conjecture is valid for $r=2$ and  $n\geq 3$. In this paper, we consider the classification of linear Lee codes with the second-best possibility, that is the density of the lattice packing of $\Z^n$ by Lee spheres $S(n,r)$ equals $\frac{|S(n,r)|}{|S(n,r)|+1}$. We show that, for $r=2$ and $n\equiv 0,3,4 \pmod{6}$,  this packing density can never be achieved.
	\end{abstract}
	\maketitle
	
\section{Introduction}
Let $\Z$ denote the ring of integers. For two words $x=(x_1,\cdots, x_n)$ and $y=(y_1,\cdots, y_n)\in \Z^n$, the \emph{Lee distance} (also known as $\ell_1$-norm, taxicab metric, rectilinear distance or Manhattan distance) between them is defined by
\[d_L(x,y)=\sum_{i=1}^n |x_i-y_i| \text{ for }x,y\in\Z^n.\]

A \emph{Lee code} $C$ is just a subset of $\Z^n$ endowed with the Lee distance. If $C$ further has the structure of an additive group, i.e.\ $C$ is a lattice in $\Z^n$, then we call $C$ a \emph{linear Lee code}. Lee codes have many practical applications, for example, constrained and partial-response channels \cite{roth_lee-metric_1994}, flash memory \cite{schwartz_quasi-cross_2012} and interleaving schemes \cite{blaum_interleaving_1998}.

The minimum distance between any two distinct elements in $C$ is called the \emph{minimum distance} of $C$. Given a Lee code of minimum distance $2r+1$, for any $x\in \Z^n$, if one can always find a unique $c\in C$ such that $d_L(x,c)\leq r$, then $\cC$ is called a \emph{perfect code}. This is equivalent to 
$$\Z^n =\dot{\bigcup}_{c\in C} (S(n,r)+c),$$
where 
\[S(n,r):= \left\{(x_1,\cdots,x_n)\in \Z^n: \sum_{i=1}^n |x_i|\leq r \right\}\] 
and $S(n,r)+c := \{v+c: v\in S(n,r) \}$. Thus, the existence of a perfect Lee code implies a tiling of $\Z^n$ by Lee spheres of radius $r$.

Perfect Lee codes exist for $n=1,2$ and any $r$, and for $n\geq 3$ and $r=1$. Golomb and Welch \cite{golomb_perfect_1970} conjectured that there are no more perfect Lee codes for other choices of $n$ and $r$. This conjecture is still far from being solved, despite many efforts and various approaches applied on it. We refer the reader to the recent survey \cite{horak_50_2018} and the references therein. For perfect codes with respect to other metrics, see the very recent monograph \cite{etzion_perfect_2022} by Etzion.

\begin{figure}[h!]
	\centering
	\includegraphics[width=0.8\linewidth]{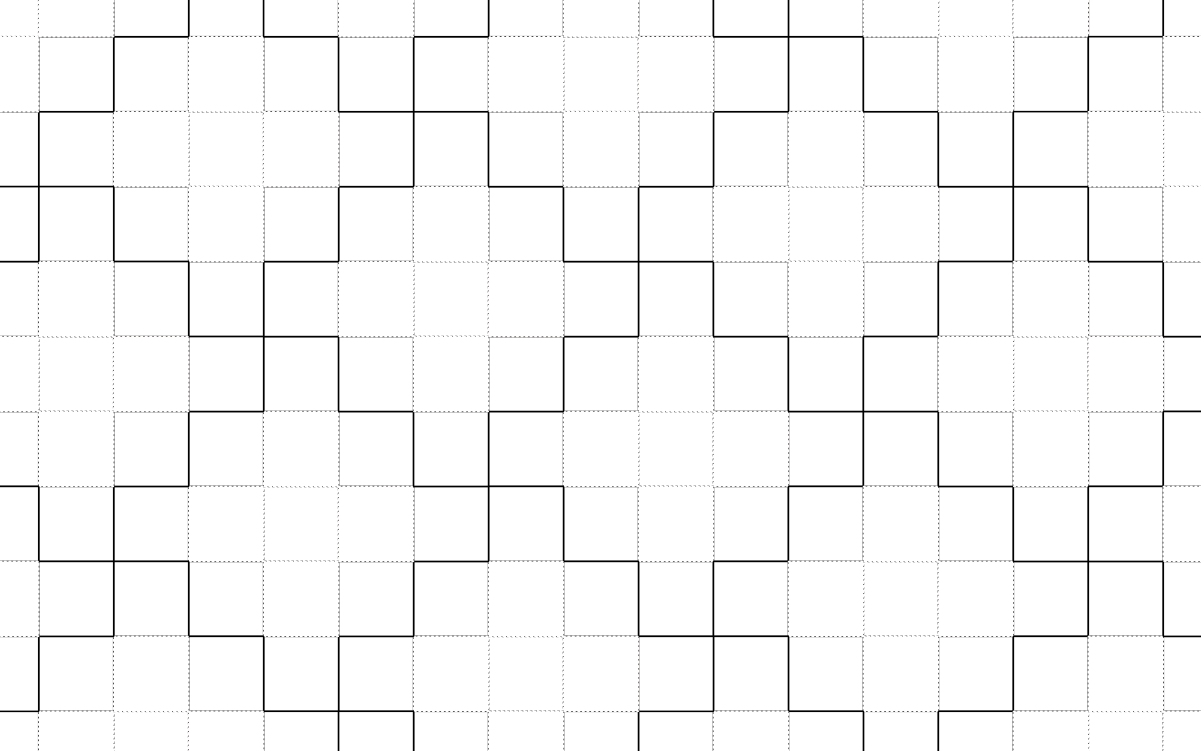}
	\caption{Tiling of $\mathbb{R}^2$ by $L(2,2)$}
	\label{fig:tiling}
\end{figure}

Let $L(n,r)$ denote the union of $n$-cubes centered at each point of $S(n,r)$ in $\mathbb{R}^n$. It is not difficult to see that there is a tiling of $\Z^n$ by $S(n,r)$ if and only if there is a tiling of $\mathbb{R}^n$ by $L(n,r)$. Figure \ref{fig:tiling} shows a (lattice) tiling of $\mathbb{R}^2$ by $L(2,2)$. As the shape of $L(n,r)$ is close to a cross-polytope when $r$ is large enough, one can use the cross-polytope packing density to prove the Golomb-Welch conjecture provided that $r$ is large enough compared with $n$. In fact, this idea was first applied by Golomb and Welch themselves in \cite{golomb_perfect_1970}. There are some other geometric approaches,  including the analysis of some local configurations of the boundary of Lee spheres in a tiling by Post \cite{post_nonexistence_1975}, and the density trick by Astola \cite{astola_elias-type_1982} and Lepist\"o \cite{lepisto_modification_1981}.

However, it seems that the geometric approaches do not work for small $r$ and large $n$. In the past several years, algebraic approaches have been proposed and applied on the existence of perfect linear Lee code for small $r$. In \cite{kim_2017_nonexistence}, Kim introduced a symmetric polynomial method to study this problem for sphere radius $r = 2$. This approach has been extended by Zhang and Ge \cite{zhang_perfect_lp_2007}, and Qureshi \cite{qureshi_nonexistence_ZGcondition_2020} for $r\geq 3$. See \cite{horak_algebraic_2016,horak_50_2018,qureshi_non-existence_2018,zhang_nonexistence_2019} for other related results. In particular, Leung and the second author \cite{leung_lattice_2020} succeeded in getting a complete solution to the case with $r=2$: there is no perfect linear Lee code of minimum distance $5$ in $\Z^n$ for $n\geq 3$. 

It is worth pointing out that the existence of a perfect linear Lee code of minimum distance $2r+1$ in $\Z^n$  is equivalent to an abelian Cayley graph of degree $2n$ and diameter $r$ whose number of vertices meets the so-called \emph{abelian Cayley Moore bound}; see \cite{camarero_quasi-perfect_lee_2016,zhang_nonexistence_2019}. For more results about the degree-diameter problems in graph theory, we refer to the survey \cite{miller_moore_2013}.

It is clear that a perfect Lee code $C$ means the packing density of $\Z^n$ by $S(n,r)$ with centers consisting of all the elements in $C$ is $1$. If there is no perfect linear Lee code for $r\geq 2$ and $n\geq 3$, one may wonder whether the second best is possible, which is about the existence of a lattice packing of $\Z^n$ by $S(n,r)$ with density $\frac{|S(n,r)|}{|S(n,r)|+1}$. 
We call such a linear Lee code \emph{almost perfect}.  In Figure \ref{fig:Z_14}, we present a lattice packing of $\Z^2$ by $S(2,2)$, and its packing density is $\frac{|S(2,2)|}{|S(2,2)|+1}=\frac{13}{14}$. This example and the numbers labeled on the cubes will be explained later in Example \ref{ex:G_n=1,2}. For convenience, we abbreviate the term almost perfect linear Lee code to \emph{APLL} code. 

\begin{figure}[h!]
	\centering
	\includegraphics[width=0.8\textwidth]{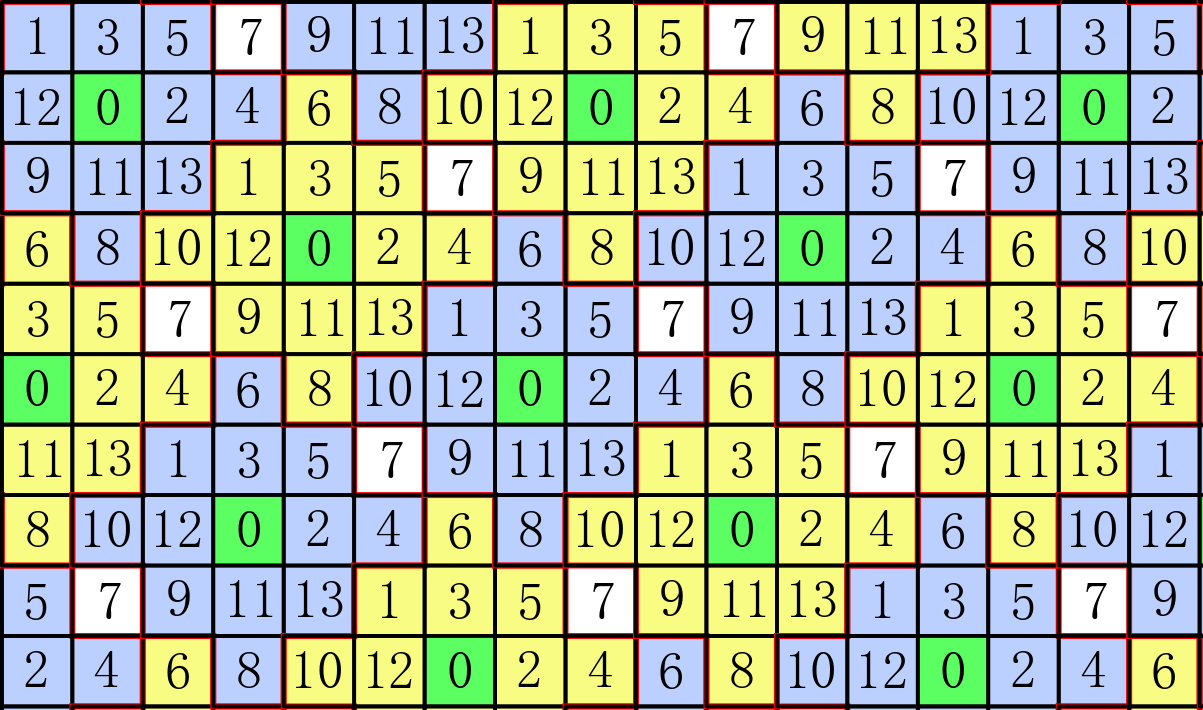}
	\caption{An almost perfect linear Lee code in $\Z^2$}
	\label{fig:Z_14}
\end{figure}

The \emph{packing radius} of a Lee code $C$ is defined to be the largest integer $r'$ such that for any element $w\in\Z^n$  there exists at most one codeword $c\in C$ with $d_L(w,c)\leq r'$. The \emph{covering radius} of a Lee code $C$ is the smallest integer $r''$ such that for any element $w\in\Z^n$  there exists at least one codeword $c\in C$ with $d_L(w,c)\leq r''$.  It is clear that a Lee code is perfect if and only if its packing radius and covering radius are the same.
For an APLL code with packing density $\frac{|S(n,r)|}{|S(n,r)|+1}$,  its 
packing radius and covering radius are $r$ and $r+1$, respectively, and its 
minimum distance  is $2r+1$ for $n>1$ and $2r+2$ for $n=1$.

A Lee code is called \emph{$t$-quasi-perfect} if its packing radius and 
covering radius are $t$ and $t+1$, respectively. Note that an APLL code is  a quasi-perfect Lee code, but a quasi-perfect Lee code is not necessarily almost perfect. For instance, two families of $2$-quasi-perfect codes in $\Z^n$ are constructed in \cite{camarero_quasi-perfect_lee_2016} and the packing density equals $\frac{2n^2+2n+1}{(2n \pm 1)^2}$, 
where $n=2[\frac{p}{4}]$ for a prime $p$ satisfying $p\geq 7$ and $p\equiv \pm 5\pmod{12}$.
In particular, the packing density equals $\frac{41}{49}$  when $n=4$, and it tends to $\frac{1}{2}$ for big $n$.  We refer to 
\cite{albdaiwi_quasi-perfect_2003,camarero_quasi-perfect_lee_2016,mesnager_2-correcting_2018} for more constructions and other results on quasi-perfect Lee codes.

About the existence of APLL codes of packing radius $2$, one can apply the symmetric polynomial method in \cite{kim_2017_nonexistence} and the algebraic number theory approach in \cite{zhang_nonexistence_2019} to derive some partial results. For $n\leq 10^5$, one can exclude the existence of APLL codes of minimum distance $5$ for $76,573$ choices of $n$. For more details, see \cite{he_nonexistence_abelian_2021}. 

The main result of this paper is the following one.
\begin{theorem}\label{th:main}
	Let $n$ be a positive integer. If $n\equiv 0,3,4\pmod{6}$, 
	then there exists no almost perfect linear Lee code of packing radius 
	$2$.
\end{theorem}

As in \cite{leung_lattice_2020}, our proof is given by investigating group ring equations in $\Z[G]$ where $G$ is of order $|S(n,2)|+1=2(n^2+n+1)$. However, the situation here is different from the one appeared in \cite{leung_lattice_2020}. We consider a special subset $T\subseteq G$ and show that $T$ splits into two disjoint subsets  $T_0\subseteq H$ and $fT_1\subseteq fH$
where $H$ is the unique subgroup of $G$ of index $2$ and $f$ is the unique involution in $G$. 
Then,  we analyze the elements appearing in $(T_0^{(2)} +T_1^{(2)})T_0$ and $(T_0^{(2)} +T_1^{(2)})T_1$. 
Unfortunately, we could not get any contradiction when $n\equiv 1,2,5 \pmod{6}$. But we conjecture that there is no APLL code for this case when $n>2$.

The rest part of this paper consists of three sections. In Section \ref{sec:group_ring}, we convert the existence of an almost perfect linear Lee code of packing radius $2$ into some conditions in group ring. Then we prove Theorem \ref{th:main} in Section \ref{sec:main}. In Section \ref{sec:concluding}, we conclude this paper with several remarks and research problems.

\section{Necessary and sufficient conditions in group ring}\label{sec:group_ring}
Let $\Z[G]$ denote the set of formal sums $\sum_{g\in G} a_g g$, where $a_g\in \Z$ and $G$ is a finite multiplicative group. The addition and the multiplication on $\Z[G]$ are defined by
$$\sum_{g\in G} a_g g +\sum_{g\in G} b_g g :=\sum_{g\in G} (a_g+b_g) g,$$
and
$$(\sum_{g\in G} a_g g )\cdot (\sum_{g\in G} b_g g) :=\sum_{g\in G} (\sum_{h\in G} a_hb_{h^{-1}g})\cdot g,$$
for $\sum_{g\in G} a_g g, \sum_{g\in G} b_g g\in\Z[G]$. Moreover,
$$\lambda \cdot(\sum_{g\in G} a_g g ):= \sum_{g\in G} (\lambda a_g) g $$
for $\lambda\in\Z$ and $\sum_{g\in G} a_g g\in \Z[G]$.

For an element $A=\sum_{g\in G} a_g g\in \Z[G]$ and $t\in \Z$, we define 
$$A^{(t)}:=\sum_{g\in G} a_g g^t.$$
A multi-subset $D$ of $G$ can be viewed as an element $\sum_{g\in D} g\in \Z[G]$. In the rest of this paper, by abuse of notation, we will use the same symbol to denote a multi-subset of $G$ and the associated element in $\Z[G]$. Moreover, $|D|$ denotes the number of distinct elements in $D$.

The following result converts the existence of APLL codes into an algebraic combinatorics problem on abelian groups. Its proof is not difficult, and more or less the same as the proof of Theorem 6 in \cite{horak_diameter_2012}.
\begin{lemma}\label{le:Z^n_to_group}
	There is an APLL code of packing radius $r$ in $\Z^n$ if and only if 
	there is an abelian group $G$ and a homomorphism $\varphi: \Z^n \rightarrow 
	G$ such that the restriction of $\varphi$ to $S(n,r)$ is injective and 
	$G\setminus \varphi(S(n,r))$ has only one element.
\end{lemma}
\begin{proof}
	First, let us prove the necessary part of the condition. Assume that $\cC$ 
	is an APLL code of packing radius $r$ and $w\in \Z^n$ is an 
	arbitrary element which is not covered by the lattice packing 
	$\dot{\bigcup}_{c\in C}(S(n,r)+c)$. Then, by the definition of APLL codes, 
	$C$ is a subgroup of $\Z^n$ and
	\begin{equation}\label{eq:Z^n=split}
		\Z^n =   \left(\dot{\bigcup}_{c\in C}(S(n,r)+c)\right) \dot{\cup} (w+C). 
	\end{equation}
	Consider the quotient group $G=\Z^n / C$ and take $\varphi$ to be the canonical homomorphism from $\Z^n$ to $G$. If there are $x,y\in S(n,r)$, such that $\varphi(x)=\varphi(y)$, then $z=x-y\in C$. Since $S(n,r)$ contains $y$,  $x\in S(n,r)+z$. This means $x$ is covered by both $S(n,r)+z$ and $S(n,r)$. By \eqref{eq:Z^n=split}, $z$ must be $0$ which implies $x=y$. Therefore, $\varphi$  is injective and $\varphi(w)$ is the unique element not in $\varphi(S(n,r))$.
	
	To prove the sufficient part, one only has to define $C=\ker(\varphi)$. The verification of the almost perfect property of $C$ is straightforward.
\end{proof}

The number of elements in a Lee sphere is
\[|S(n,r)|=\sum_{i=0}^{\text{min}\{n,r\}}2^{i}\binom {n}{i}\binom {r}{i},\]
which can be found in \cite{golomb_perfect_1970} and \cite{stanton_note_1970}. In particular, when $r=2$, $|S(n,r)|=2n^2+2n+1$ and the group $G$ considered in Lemma \ref{le:Z^n_to_group} is of order $2n^2+2n+2$.

The next result translates the existence of an APLL code into a group ring condition. The same result has been proved in \cite{he_nonexistence_abelian_2021} in the context of the existence of abelian Moore Cayley graphs with excess one. For completeness, we include a proof here.

\begin{lemma}\label{le:group_ring_inter}
	There exists an APLL code of packing radius $2$ in $\Z^n$ if and 
	only if there is an abelian group $G$ of order $2(n^2+n+1)$ and an 
	inverse-closed subset $T\subseteq G$ containing $e$ with $|T|=2n+1$ such 
	that 
	\begin{equation}\label{eq:main_T^2}
	T^2 =2(G-f)-T^{(2)}+2n e,
	\end{equation}
	where  $e$ is the identity element in $G$ and $f$ is the unique element of order $2$ in $G$.
\end{lemma}
\begin{proof}
	The proof of \eqref{eq:main_T^2} is similar to the proof of Lemma 2.3 in 
	\cite{zhang_nonexistence_2019}. Suppose that there exists an APLL code of packing radius $2$. By Lemma \ref{le:Z^n_to_group}, we have a 
	homomorphism $\varphi:\Z^n \rightarrow G$. Define $a_i = \varphi(e_i)$ for 
	$i=1,\cdots, n$.  As the radius equals $2$, by definition, $G$ has the 
	following partition
	\[ G=\{e,f\} \dot{\bigcup} \{a_i^{\pm 1}, a_i^{\pm 2}: i= 1,\cdots, n\} \dot{\bigcup} \{a_i^{\pm 1}a_j^{\pm 1}: 1\leq i<j\leq   n\},\]
	where $f$ is the unique element of $G$ not in $\varphi(S(n,2))$.
	
	Let $T=e+ \sum_{i=1}^n a_i + \sum_{i=1}^n a_i^{-1}$. Hence $|T|=2n+1$. Moreover,
	\begin{align*}
		T^2 &= e+ 2\sum_{i=1}^n a_i + 2\sum_{i=1}^n a_i^{-1}+\left(\sum_{i=1}^n a_i\right)^2 +\left(\sum_{i=1}^n a_i^{-1}\right)^2 +2\left(\sum_{i=1}^n a_i\right)\left(\sum_{i=1}^n a_i^{-1}\right)\\
		&=(2n+1)e + 2 \left(\sum_{i=1}^n a_i + \sum_{i=1}^n a_i^{-1}\right) + \sum_{i=1}^n (a_i^2+a_i^{-2}) + 2\sum_{1\leq i<j\leq n}a_i^{\pm 1} a_j^{\pm 1}\\
		&= 2ne + 2(G-f) - T^{(2)}.
	\end{align*}
	Therefore \eqref{eq:main_T^2} is proved.
	
	Applying the map $x\mapsto x^{(-1)}$ on the both sides of \eqref{eq:main_T^2}, we get 
	\[ T^2 = 2(G-f^{-1}) - T^{(2)} + 2ne,\]
	because $T$ is inverse-closed. By comparing this equation with \eqref{eq:main_T^2}, we see $f=f^{-1}$. As $n^2+n+1$ is always odd, there is only one involution in $G$ which is $f$. 
	
	Next we prove the sufficiency part. Given an inverse-closed subset $T\subseteq G$ containing $e$ with $|T|=2n+1$ satisfying $\eqref{eq:main_T^2}$, we define a homomorphism $\varphi:\Z^n\rightarrow G$ by $\varphi(e_i)=a_i$. It is clear that  $\varphi$ is completely determined by the value of all $a_i$'s. As $T$ is inverse-closed and $|T|=2n+1$, the restriction of $\varphi$ to $S(n,1)$ is injective. Furthermore, by checking the coefficients of elements in \eqref{eq:main_T^2} as in the necessity part of the proof, we see that
	\[G\setminus \{f\} = \{e\}\dot{\bigcup} \{a_i^{\pm 1}, a_i^{\pm 2}: i= 1,\cdots, n\} \dot{\bigcup} \{a_i^{\pm 1}a_j^{\pm 1}: 1\leq i<j\leq   n\},\]
	that is $G\setminus \{f\}=\varphi(S(n,2))$. By Lemma \ref{le:Z^n_to_group}, 
	$C=\ker(\varphi)$ is  an APLL code of packing radius $2$.
\end{proof}

\begin{remark}
	Similar to the alternative proof of the uniqueness of $f\in G$ in Lemma 
	\ref{le:group_ring_inter}, we can show that if there is an APLL code of 
	packing radius $r$, then there is a group $G$ of even order with a 
	unique involution in $G$ by checking $a_{i_1}\cdots 
	a_{i_j}=a_{i_1}^{-1}\cdots a_{i_j}^{-1}$ for $j\leq r$. However, the 
	corresponding group ring equations become more complicated, and we are not 
	going to investigate them in this paper.
\end{remark}

The following two examples show that for $n=1,2$, there do exist APLL codes of packing radius $2$. The corresponding result of Example \ref{ex:G_n=1,2} 
(b) has been already depicted in Figure \ref{fig:Z_14}.
\begin{example}\label{ex:G_n=1,2}
	Let $C_m$ denote the cyclic group generated by $g$ of order $m$.
	\begin{enumerate}[label=(\alph*)]
		\item $n=1$, $G=C_6$ and $T=\{e, g^{\pm 1}\}$.
		\item $n=2$, $G=C_{14}$ and $T=\{e, g^{\pm 2}, g^{\pm 3}\}$.
	\end{enumerate}
	In Figure \ref{fig:Z_14}, we label each element in $\Z^2$ which is mapped 
	to $\varphi(e_1)=g^2$ by $2$ and those mapped to $\varphi(e_2)=g^3$ 
	by $3$. The center of every Lee sphere is labeled by $0$. The holes 
	are all labeled by $7$ which corresponds to the unique involution $g^7$ in 
	$G$.
	
	In (a) and (b), $G$  has a unique subgroup $H$ of order $n^2+n+1$. We will obtain further necessary and sufficient conditions in terms of subsets in $H$ in the later part, and the corresponding results for (a) and (b) will be given in Example \ref{ex:H_n=1,2}.
\end{example}

\begin{remark}\label{re:ex_13_14}
	It is easy to check that $\{e, g^{\pm 3}, g^{\pm 2}\}\subseteq 
	C_{13}$ defines a perfect linear Lee code in $\Z^2$ as a lattice with basis $\{(2,3),(-3,2)\}$; see Figure \ref{fig:tiling}. Accordingly the basis of the APLL code in Figure \ref{fig:Z_14} generated by $\{e, g^{\pm 3}, g^{\pm 2}\}\subseteq C_{14}$ is $\{(4,2), (3,-2)\}$.

	The isometry group of $\Z^n$ is generated by permutations of axes and 
	reflections; see \cite{horak_50_2018}. Consequently, all APLL codes of 
	packing radius $2$ in $\Z^2$ are isometric for the following reasons. By Lemma \ref{le:group_ring_inter}, $T$ must contain a primitive element of 
	$C_{14}$. Moreover, it is easy to check that if we 	assume that $e, g^{\pm 1}\in T$, then $T$ must be $\{e, g^{\pm 1}, g^{\pm 4}\}$. Thus, all $T$ satisfying Lemma 2.2 must be of the shape $\{e, g^{\pm 1}, g^{\pm 	4}\}^\psi$ for 	some group isomorphism $\psi$ of $C_{14}$. 
	Finally,  note that $T$ and $T^\psi$ provide the same Lee code in $\Z^2$ 
	for any group isomorphism $\psi$. Therefore, all APLL codes of packing 
	radius $2$ in $\Z^2$ must be isometric.
\end{remark}

In Lemma \ref{le:group_ring_inter}, we have obtained some necessary and 
sufficient condition for the existence of APLL codes of packing radius 
$2$ in terms of a subset $T$ in an abelian group $G$ of order $2(n^2+n+1)$. We 
can further split $T$ into subsets  in a subgroup $H$  and its coset $fH$.

Let $H$ denote the unique subgroup of order $n^2+n+1$ in $G$. Define $T_0=T\cap H$ and $T_1=fT\cap H$. Thus
\[  T=T_0 +fT_1.\]
By \eqref{eq:main_T^2}, 
\[ T_0^2+T_1^2+2fT_0T_1 = 2(G-f)-T_0^{(2)}-T_1^{(2)} +2n\cdot e. \]
By considering  the above equation intersecting $H$ and $fH$, respectively, we obtain
\begin{eqnarray}
\label{eq:T0T1} T_0T_1&=&H-e,\\
\label{eq:T0^2+T1^2} T_0^2+T_1^2&=&2H-T_0^{(2)}-T_1^{(2)}+2ne.
\end{eqnarray}
Therefore, we have proved the following result.
\begin{lemma}\label{le:T0_T1}
	Let $T=T_0+fT_1\subseteq G$ with $T_0$ and $T_1\subseteq H$. The subset $T$ satisfies that  $e\in T$, $T^{(-1)}=T$ and \eqref{eq:main_T^2} if and only if $e\in T_0$, $T_0^{(-1)}=T_0$, $T_1^{(-1)}=T_1$,  \eqref{eq:T0T1} and \eqref{eq:T0^2+T1^2} hold.
\end{lemma}

\begin{example}\label{ex:H_n=1,2}
	The following examples  of $T_0$ and $T_1$  are derived from $T$ in Example \ref{ex:G_n=1,2} (a) and (b), respectively. It is easy to check that $T_0$ and $T_1$ satisfy \eqref{eq:T0T1} and \eqref{eq:T0^2+T1^2}.
	\begin{enumerate}[label=(\alph*)]
		\item For $n=1$, $G=C_6=\langle g \rangle$ and $H=2C_{6}=\langle g^2\rangle \cong C_3$. Let $T=\{e, g^{\pm 1}\}$. Then $T_0=\{e\}$ and $T_1=\{g^2,g^4\}$.
		\item 	For $n=2$, $G=C_{14}=\langle g \rangle$ and $H=2C_{14}=\langle g^2\rangle \cong C_7$.  From $T=\{e , g^{\pm2}, g^{\pm3}\}$, one gets
		\[ T_0=\{e,g^2,g^{12}\} \text{ and }T_1=\{g^4,g^{10}\}.\]
	\end{enumerate}
\end{example}

\begin{remark}
	A \emph{near-factorization} of a group $H$ is a pair of (not necessarily inverse-closed) subsets $T_0$ and $T_1$ such that \eqref{eq:T0T1} holds. There are very little known results on them. However, it has been proved that, up to translate, $T_0$ and $T_1$ must be inverse-closed; see \cite{caen_near_1990}.
\end{remark}

The following result is a collection of obvious necessary conditions for $T_0$ and $T_1$, which will be intensively used in Section \ref{sec:main}.
\begin{lemma}\label{le:T_0T_1_nece_condi}
	Suppose that $T=T_0+fT_1\subseteq G$ satisfying $e\in T$, $T^{(-1)}=T$, $|T|=2n+1$, \eqref{eq:T0T1} and \eqref{eq:T0^2+T1^2}.  Then the following statements hold.
	\begin{enumerate}[label=(\alph*)]
		\item $e\in T_0$, $e\notin T_1$;
		\item $T_0 \cap T_1=\emptyset$ and $T_0^{(2)} \cap T_1^{(2)}=\emptyset$;
		\item $T_0\cap (T_0^{(2)}\setminus \{e\})=T_0\cap T_1^{(2)}=\emptyset$;
		\item $\{ab: a\neq b,  a,b\in T_0\}\cap T_0^{(2)} =\{e\}$; 
		\item When $n$ is odd, $|T_0|= n$ and $|T_1|=n+1$;
		\item When $n$ is even, $|T_0|= n+1$ and $|T_1|=n$;
		\item There is no common non-identity element in $T_0^2$ and $T_1^2$;
		\item $T_0\cap T_0^{(3)} = \{e\}$.
	\end{enumerate}
\end{lemma}
\begin{proof}
	By the definition of $T$ and $T_0$, we have (a). From \eqref{eq:T0T1} and the inverse-closed property of $T_1$, we derive $T_0 \cap T_1=\emptyset$. As $\gcd(2, |H|)=1$,  the map $x\mapsto x^2$ on $H$ is a group automorphism. Thus $T_0^{(2)} \cap T_1^{(2)}=\emptyset$ and (b) is proved.
	
	Note that $T_0$, $T_0^{(2)}\setminus\{e\}$ and $T_1^{(2)}$ are all subsets in $H$.  By checking \eqref{eq:T0^2+T1^2}, we know that they cannot share any common element. By the same reason, we also prove (d).
	
	From \eqref{eq:T0T1}, we get $|T_0|\cdot |T_1| = n(n+1)$. Notice that $|T_0|$ must be odd because it has an involution $x\mapsto x^{-1}$ with a unique fixed element $e$. Together with $|T_0|+|T_1|=2n+1$, we derive (e) and (f). By checking the elements appeared in $T_0^2+T_1^2$ in  \eqref{eq:T0^2+T1^2}, we get (g).
	
	Finally, we prove (h). It is clear that $e\in T_0\cap T_0^{(3)}$. Assume to the contrary that there exists $b\in T_0\setminus \{e\}$ such that $b=a^3$ for some $a\in T_0$. Then $ a^2=a^{-1}b\in T_0^2 $. As $a^{-1}\neq b$, this is a contradiction to (d).
\end{proof}

In \cite{he_nonexistence_abelian_2021}, several different approaches have been applied to derive various nonexistence results. In particular, we need the following one in the next section.
\begin{proposition}[Corollary 3.1 in \cite{he_nonexistence_abelian_2021}]
	\label{prop:known}
	Suppose that $8n-7$ is not a square in $\Z$ and one of the following conditions are satisfied:
	\begin{itemize}
		\item $3,7,19$ or $31$ divides $n^2+n+1$;
		\item $13 \mid n^2+n+1$ and $8n-11\notin \{13k^2 : k\in\Z \}$.
	\end{itemize}
	Then there is no  inverse-closed subset $T\subseteq G$ satisfying Lemma \ref{le:group_ring_inter}.
\end{proposition}
Although Proposition \ref{prop:known} does not work for $n=3$, this situation can be excluded by a direct verification because the group $G$ is small.

The symmetric polynomial method applied by Kim \cite{kim_2017_nonexistence} on the existence of perfect linear Lee codes can also be used directly here to derive strong nonexistence results; see Theorem 3.1 in \cite{he_nonexistence_abelian_2021}. In particular, it leads to the following results appeared in \cite{he_nonexistence_abelian_2021}.
\begin{result}\label{result:prime}
	For $3<n\leq 10^5$, if $n^2+n+1$ is a prime, then  there is no  inverse-closed subset $T\subseteq G$ satisfying Lemma \ref{le:group_ring_inter}.
\end{result}

\section{Proof of the main result}\label{sec:main}

In this section, we are going to prove Theorem \ref{th:main} by showing the nonexistence of inverse-closed subsets $T_0,T_1\subseteq H$ satisfying $e\in T_0$, $|T_0|+|T_1|=2n+1$, \eqref{eq:T0T1} and \eqref{eq:T0^2+T1^2}.

Define $\hat{T}=T_0+T_1\in \Z[H]$. Write $\hat{T}^{(2)}=\sum_{i=0}^{2n} a_i$ where $a_0=e$, $a_i\in T_0^{(2)}$ for $i=0,\cdots, |T_0|-1$ and $a_i\in T_1^{(2)}$ for $i= |T_0|,\cdots, 2n$.  Let $k_0=|T_0|$ and $k_1=|T_1|$. 

By multiplying $T_0$ and $T_1$ on the both sides of \eqref{eq:T0^2+T1^2}, respectively, and rearranging the terms, we get
\begin{align*}
	\hat{T}^{(2)}T_0 &= (2k_0-k_1)H+T_1-T_0^3+2nT_0,\\
	\hat{T}^{(2)}T_1 &= (2k_1-k_0)H+T_0-T_1^3+2nT_1.
\end{align*}
Consider the above two equations modulo $3$ 
\begin{align*}
	\hat{T}^{(2)}T_0 &\equiv (2k_0-k_1)H+T_1-T_0^{(3)}+2nT_0 \pmod{3},\\
	\hat{T}^{(2)}T_1 &\equiv (2k_1-k_0)H+T_0-T_1^{(3)}+2nT_1 \pmod{3}.
\end{align*}
Note that $T_0^3 \equiv T_0^{(3)} \pmod{3}$. If $3\nmid |H|$, then $T_0^{(3)}$ is a subset; for $3\mid |H|$, we will show that there are only very few elements in $T_0^{(3)}$ appearing more than once in Lemma \ref{le:n=1mod3_rep}. 
Thus, we can derive some strong conditions on the coefficients of most of the elements in the right-hand side of the above two equations. For instance, when $n\equiv 1 \pmod{3}$ and $n$ is odd, the first equation becomes
\[
	\hat{T}^{(2)}T_0 \equiv T_1-T_0^{(3)}+2T_0 \pmod{3}.
\]
As $T_1$, $T_0^{(3)}$ and $T_0$ are approximately of size $n$, most of the elements in $H$ appear in $\hat{T}^{(2)}T_0$ for $3k$ times, $k=0,1,\cdots$.


Let $X_i$ ($Y_i$, resp.) be the subset of elements of $H$ appearing in $\hat{T}^{(2)} T_0$ ($\hat{T}^{(2)} T_1$, resp.) exactly $i$ times for $i=0,1,\cdots, M_0$ ($M_1$, resp.), which means $X_i$'s ($Y_i$'s, resp.) form a partition of the group $H$. In particular, we define $M_0$ and $M_1$ to be the largest integers such that $X_{M_0}$ and $Y_{M_1}$ are non-empty sets. Then
\begin{align}
	\label{eq:T^2T_0} \hat{T}^{(2)} T_0 &= \sum_{i=0}^{M_0} iX_i,\\
	\label{eq:T^2T_1} \hat{T}^{(2)} T_1 &= \sum_{i=0}^{M_1} iY_i.
\end{align}
By \eqref{eq:T^2T_0} and \eqref{eq:T^2T_1}, we have the following three conditions on the value of $|X_i|$'s and $|Y_i|$'s:
\begin{align}
\label{eq:sum_i|X_i|}	\sum_{i=1}^{M_0}i|X_i| &=(2n+1)k_0,\\
\label{eq:sum_i|Y_i|}	\sum_{i=1}^{M_1}i|Y_i| &=(2n+1)k_1,\\
\label{eq:sum_|X_i|}	\sum_{i=0}^{M_0}|X_i|=\sum_{i=0}^{M_1}|Y_i|&= n^2+n+1.
\end{align}
Some extra conditions are given by the following Lemma.
\begin{lemma}\label{le:in-ex_X_Y}
	For $X_i$ and $Y_i$ defined in \eqref{eq:T^2T_0} and \eqref{eq:T^2T_1},
	\begin{align}
	\label{eq:sum_|X_i|-0} \sum_{i=1}^{M_0}|X_i|=&(2n+1)k_0-2(k_0-1)k_0+\theta_0 + \sum_{s\geq 3} \frac{(s-1)(s-2)}{2}|X_s|,\\
	\label{eq:sum_|Y_i|-0}  \sum_{i=1}^{M_1}|Y_i|=&(2n+1)k_1-2(k_1-1)k_1 + \theta_1+ \sum_{s\geq 3} \frac{(s-1)(s-2)}{2}|Y_s|,
	\end{align}
	where $\theta_0=|(T_0^2 \setminus T_0^{(2)})\cap \hat{T}^{(4)}|$ and $\theta_1 = \frac{|T_1\cap \hat{T}^{(2)}|}{2}+|(T_1^2\setminus (T_1^{(2)} \cup \{e\}))\cap \hat{T}^{(4)}|$. 
\end{lemma}
Note that $|X_i|$'s and $|Y_i|$'s are nonnegative integers. Our main idea is to use \eqref{eq:sum_i|X_i|}	, \eqref{eq:sum_i|Y_i|}	, \eqref{eq:sum_|X_i|}, \eqref{eq:sum_|X_i|-0} and \eqref{eq:sum_|Y_i|-0} together with $\hat{T}^{(2)}T_0$ and $\hat{T}^{(2)}T_1 \pmod{3}$ to determine $|X_i|$'s and $|Y_i|$'s. For some cases, we will end up with a contradiction which means there is no $T_0$ and $T_1$ satisfying \eqref{eq:T0T1} and \eqref{eq:T0^2+T1^2}, for example, see Theorem \ref{th:n=1_even}; for the other cases, the value of $|X_i|$'s and $|Y_i|$'s together with a careful analysis of \eqref{eq:T0T1} and \eqref{eq:T0^2+T1^2} also lead to contradictions, for instance, see Theorems \ref{th:n=0_odd} and \ref{th:n=0_even}.

To prove Lemma \ref{le:in-ex_X_Y}, we need  the following fact.
\begin{lemma}\label{le:e_in_X_1}
	For $X_i$'s defined by \eqref{eq:T^2T_0}, $e\in X_1$, $|X_i|$ is even for $i\neq 1$ and $|X_1|$ is odd .
\end{lemma}
\begin{proof}
	The identity element $e$ is in $X_1$ for the following reasons: First, $e\notin T_1^{(2)}T_0$, otherwise $ T_0\cap T_1^{(2)}\neq \emptyset$ which contradicts Lemma \ref{le:T_0T_1_nece_condi} (c); second, $e$ appears in  $T_0^{(2)}T_0$ exactly once which also follows from $T_0^{(2)}\cap T_0=\{e\}$ by Lemma \ref{le:T_0T_1_nece_condi} (c).
	
	Note that $X_i=X_i^{(-1)}$ for $i\geq 1$. Hence $|X_{i}|$ is even for $i\geq 2$, and $|X_1|$ is odd because $e\in X_1$. Moreover, as $|H|=n^2+n+1$ is odd and $X_i$'s form a partition of $H$, $|X_0|$ is  even.
\end{proof}

Recall that we have defined $k_0=|T_0|$ and $k_1=|T_1|$ at the beginning of this section.
\begin{lemma}\label{le:C_k,l}
	Denote the elements in $T_0^{(2)}$ by $a_i$ with $i=0,\cdots, k_0-1$ and the elements in $T_1^{(2)}$ by $a_i$ with $i=k_0,\cdots, 2n$. In particular, let $a_0=e$.
	Then
	\[
	 |a_i T_k\cap a_j T_k| \in \{0,1,2\},
	\]
	for any $i<j$ and $k=0,1$.
	For $\ell=0,1,2$, define 
	\[
	C_{k,\ell}=\#\left\{ (a_i,a_j):  i<j, |a_i T_k\cap a_j T_k| =\ell \right\}.
	\]
	Then 
	\begin{equation*}
		C_{k,1}=\begin{cases}
			2k_0-2, & k=0;\\[0.5em]
			2k_1-\frac{|T_1 \cap \hat{T}^{(2)}|}{2}, & k=1,
		\end{cases}
	\end{equation*}
	and 
	\begin{equation*}
		C_{k,2}=\begin{cases}
		(k_0-1)^2-\frac{\left|(T_0^2 \setminus T_0^{(2)})\cap \hat{T}^{(4)} \right|}{2}, & k=0;\\[0.5em]
		k_1^2-2k_1-\frac{\left|(T_1^2\setminus (T_1^{(2)} \cup \{e\}))\cap \hat{T}^{(4)}\right|}{2}, & k=1.
		\end{cases}
	\end{equation*}
\end{lemma}
\begin{proof}
	An element  $a_ix=a_jy\in a_i T_k\cap a_j T_k$ for some $x,y\in T_k$ if and only if
	\[ a_i a_j^{-1}=x^{-1}y\in T_k^2. \]
	By \eqref{eq:T0^2+T1^2}, there exist at most two possible choices of $(x,y)$. Moreover, if $(x,y)$ is such that $a_ix=a_jy$, then $a_iy^{-1}=a_jx^{-1}$ which also provides a solution $(y^{-1}, x^{-1})$. Thus $|a_i T_k\cap a_j T_k| =1$ if and only if $x=y^{-1}$, which is equivalent to $a_ia_j^{-1}\in T_k^{(2)}$.
	
	Next let us determine the value of $C_{k,\ell}$ for $\ell=1,2$. 
	\medskip
	
	\textbf{Calculation of $C_{0,1}$}.
	By the previous analysis, for $i<j$,  $|a_i T_0\cap a_j T_0| =1$ if and only if $a_ia_j^{-1}\in T_0^{(2)}\setminus \{e\}$, which is equivalent to $\sqrt{a_i}\sqrt{a_j^{-1}}\in T_0\setminus \{e\}$.
	
	If $0\leq i<j\leq k_0$, i.e.\ $\sqrt{a_i}, \sqrt{a_j}\in T_0$,  then $\sqrt{a_i}$ must be $e$ because $e\in T_0=e\cdot T_0$ which is a subset of $T_0^2$, and Lemma \ref{le:T_0T_1_nece_condi} (c) and \eqref{eq:T0^2+T1^2} tell us that $T_0 \cap \left\{ab :   a,b\in T_0\setminus\{e\}, a\neq b \right\} = \emptyset$. If $k_0\leq i<j\leq 2n$, then by Lemma \ref{le:T_0T_1_nece_condi} (g) and $T_0\subseteq T_0^2$,  $\sqrt{a_i}, \sqrt{a_j}$ cannot belong to $T_1$.
	Therefore, we have proved that  for $i<j$,  $|a_i T_0\cap a_j T_0| =1$ if and only if one of the following cases happens:
	\begin{enumerate}[label=(\alph*)]
		\item $a_i=e$ and $\sqrt{a_j}\in T_0\setminus\{e\}$;
		\item $\sqrt{a_i}\in T_0$ and $\sqrt{a_j}\in T_1$ such that $\sqrt{a_i a_j^{-1}}\in T_0$.
	\end{enumerate}
	By \eqref{eq:T0T1}, there are totally $k_0-1$ ordered pairs $(a_i, a_j)$ satisfying (b). Thus, $C_{0,1}=2(k_0-1)$.
	
		\medskip
	
	\textbf{Calculation of $C_{1,1}$}.
	Similarly, for $i<j$,  $|a_i T_1\cap a_j T_1| =1$ if and only if $a_ia_j^{-1}\in T_1^{(2)}$, which is equivalent to one of the following cases:
	\begin{enumerate}[label=(\alph*)]
		\item $a_i=e$ and $\sqrt{a_j}\in T_1$;
		\item $\sqrt{a_i},\sqrt{a_j}\in T_k$ for the same $k=0$ or $1$ such that $\sqrt{a_i a_j^{-1}}\in T_1$.
	\end{enumerate}
	The first case is obtained by checking the coefficients of elements in $T_0T_1$ via \eqref{eq:T0T1} and the second case is from \eqref{eq:T0^2+T1^2}.
	It is clear that Case (a) provides $k_1$ ordered pairs $(a_i, a_j)$ with $i<j$. 
	
	To count this number for Case (b), without loss of generality, we order the elements in $T_0$ and $T_1$ such that $a_i^{-1} = a_{k_0-i}$ for $i=1,\cdots, k_0$ and $a_{k_0+i}^{-1} = a_{2n-i}$ for $i=0,\cdots, k_1-1$. 
	
	Given an element $c\in T_1$, by \eqref{eq:T0^2+T1^2}, there always exists exactly one unordered pair $\{a,b\}\subseteq T_0$ or $T_1$ such that $ab=c$. Suppose that $a=\sqrt{a_u}$, $b=\sqrt{a_v^{-1}}$ and $a,b\in T_0$.  As $c\neq e$, $u$ and $v$ must be different. Consequently,
	\[ c=ab=\sqrt{a_u} \sqrt{a_v^{-1}}=\sqrt{a_{k_0-v}}\sqrt{a_{k_0-u}^{-1}}.  \]
	
	Note that $u<v$ implies $k_0-v<k_0-u$. 
	Thus, if $a_u\neq a_{k_0-v}$, 
	\[
	\#\left\{	(\sqrt{a_i}, \sqrt{a_j}) \in T_0\times T_0: c=\sqrt{a_i a_j^{-1}}, i<j\right\}=
	\begin{cases}
		2, & u<v;\\
		0, & u>v.
	\end{cases}
	\]
	If $a_u = a_{k_0-v}$,  then $a_u=a_v^{-1}$ and $c$ is represented by exactly one ordered pair which implies $c=a_u\in T_0^{(2)}$.
	
	Therefore, there are totally 
	\[
	2\cdot \frac{|T_1\cap (T_0^2 \setminus T_0^{(2)})|}{2}+\frac{|T_1\cap T_0^{(2)}|}{2} 
	= 
	|T_1\cap T_0^2 | - |T_1\cap T_0^{(2)}|+\frac{|T_1\cap T_0^{(2)}|}{2}
	\]
	ordered pair $(\sqrt{a_i}, \sqrt{a_j})\in T_0\times T_0$ with $i<j$ such that $\sqrt{a_i a_j^{-1}}\in T_1$. A similar analysis for $a,b\in T_1$ leads to 
	\[
	2\cdot \frac{|T_1\cap (T_1^2 \setminus T_1^{(2)})|}{2}+\frac{|T_1\cap T_1^{(2)}|}{2} 
	= 
	|T_1\cap T_1^2 | - |T_1\cap T_1^{(2)}|+\frac{|T_1\cap T_1^{(2)}|}{2}
	\]
	ordered pair $(\sqrt{a_i}, \sqrt{a_j})\in T_1\times T_1$ with $i<j$ such that $\sqrt{a_i a_j^{-1}}\in T_1$. As \eqref{eq:T0^2+T1^2} tells us that $T_1$ is covered by the elements in $T_0^2$ and $T_1^2$, by the above counting results, 
	\[C_{1,1}= k_1+\left(|T_1|-\frac{|T_1 \cap T_0^{(2)}|}{2}-\frac{|T_1 \cap T_1^{(2)}|}{2}\right)
	=2k_1-\frac{|T_1 \cap \hat{T}^{(2)}|}{2}.\]

	\textbf{Calculation of $C_{0,2}$}.
	For $i<j$,  $|a_i T_0\cap a_j T_0| =2$ if and only if $a_ia_j^{-1}\in T_0^2\setminus T_0^{(2)}$, which is equivalent to one of the following cases:
	\begin{enumerate}[label=(\alph*)]
		\item $\sqrt{a_i}\in T_0$ and $\sqrt{a_j}\in T_1$ such that $a_ia_j^{-1}\in T_0^2\setminus T_0^{(2)}$;
		\item $\sqrt{a_i},\sqrt{a_j}\in T_k$ for the same $k=0$ or $1$ such that  $a_i a_j^{-1}\in  T_0^2\setminus T_0^{(2)}$.
	\end{enumerate}
	By \eqref{eq:T0T1}, $ T_0^{(2)}T_1^{(2)}=H-e$. Thus the cardinality of the ordered pairs $(a_i, a_j)$ satisfying (a) and $i<j$  is
	\begin{align*}
		|T_0^2\setminus T_0^{(2)}|&=\frac{1}{2}\left(\#\{(a,b): a,b\in T_0\} - \#\{(a,a^{-1}): a\in T_0\} - \#\{(a,a): a\in T_0, a\neq e\} \right)\\
		&=\frac{1}{2}(k_0^2-k_0-(k_0-1)) =\frac{(k_0-1)^2}{2}.
	\end{align*}
	
	For (b), we follow the argument in the calculation of $C_{1,1}$.
	For any $c \in T_0^2\setminus T_0^{(2)}$, suppose that $a=\sqrt{a_u}$, $b=\sqrt{a_v^{-1}}$ and $a,b\in T_0$ such that $c=a_ua_v^{-1}=(ab)^2\in (T_0^{(2)})^2$. As $c\neq e$, $u$ and $v$ must be different. Consequently,
	\[ c={a_u} {a_v^{-1}}={a_{k_0-v}}{a_{k_0-u}^{-1}}.  \]
	
	It follows that
		\[
	\#\left\{	(\sqrt{a_i}, \sqrt{a_j}) \in T_0\times T_0: c=a_i a_j^{-1}, i<j\right\}=
	\begin{cases}
	2, & a_u \neq a_{k_0-v}, u<v;\\
	0, & a_u \neq a_{k_0-v}, u>v;\\
	1, & a_u = a_{k_0-v}.
	\end{cases}
	\]
	Therefore, there are totally 
	\begin{align*}
		&2\cdot \frac{|(T_0^2 \setminus T_0^{(2)})\cap ((T_0^{(2)})^2\setminus T_0^{(4)})|}{2}+\frac{|(T_0^2 \setminus T_0^{(2)})\cap T_0^{(4)}|}{2}\\
		=& 		2\cdot \frac{|(T_0^2 \setminus T_0^{(2)})\cap (T_0^{(2)})^2|}{2}-\frac{|(T_0^2 \setminus T_0^{(2)})\cap T_0^{(4)}|}{2} 
	\end{align*}
	ordered pair $(\sqrt{a_i}, \sqrt{a_j})\in T_0\times T_0$ with $i<j$ such that $a_i a_j^{-1}\in T_0^{2}\setminus T_0^{(2)}$. 
	
	A similar analysis for $a,b\in T_1$ leads to 
	\[
	2\cdot \frac{|(T_0^2 \setminus T_0^{(2)})\cap (T_1^{(2)})^2|}{2}-\frac{|(T_0^2 \setminus T_0^{(2)})\cap T_1^{(4)}|}{2}
	\]
	ordered pair $(\sqrt{a_i}, \sqrt{a_j})\in T_1\times T_1$ with $i<j$ such that $a_i a_j^{-1}\in T_0^{2}\setminus T_0^{(2)}$. As \eqref{eq:T0^2+T1^2} tells us that 
	\[
	(T_0^{(2)})^2+(T_1^{(2)})^2 = 2H-T_0^{(4)}-T_1^{(4)}+2ne,
	\]
	which means that
	$T_0^2\setminus T_0^{(2)}$ is covered by the elements in $(T_0^{(2)})^2$ and $(T_1^{(2)})^2$. By the counting results for cases (a) and (b), 
	\begin{align*}
		C_{0,2} &= |T_0^2\setminus T_0^{(2)}|+\left(|T_0^2\setminus T_0^{(2)}|- \frac{|(T_0^2 \setminus T_0^{(2)})\cap \hat{T}^{(4)}|}{2}\right)\\
		&=(k_0-1)^2- \frac{|(T_0^2 \setminus T_0^{(2)})\cap \hat{T}^{(4)}|}{2}.
	\end{align*}
	
%
%
	\medskip
	\textbf{Calculation of $C_{1,2}$}.
	For $i<j$,  $|a_i T_1\cap a_j T_1| =2$ if and only if $a_ia_j^{-1}\in T_1^2\setminus  (T_1^{(2)} \cup \{e\})$, which is equivalent to one of the following cases:
	\begin{enumerate}[label=(\alph*)]
		\item $\sqrt{a_i}\in T_0$ and $\sqrt{a_j}\in T_1$ such that $a_ia_j^{-1}\in T_1^2\setminus  (T_1^{(2)} \cup \{e\})$;
		\item $\sqrt{a_i},\sqrt{a_j}\in T_k$ for the same $k=0$ or $1$ such that  $a_i a_j^{-1}\in  T_1^2\setminus  (T_1^{(2)} \cup \{e\})$.
	\end{enumerate}
	For Case (a), \eqref{eq:T0T1} tells us that each element in $T_1^2\setminus (T_1^{(2)} \cup \{e\})$ is covered by elements in $T_0T_1$ exactly once. As a consequence, the cardinality of ordered pairs $(a_i, a_j)$ satisfying (a) and $i<j$  is
		\begin{align*}
		|T_1^2\setminus (T_1^{(2)} \cup \{e\})|&=\frac{1}{2}\left(\#\{(a,b): a,b\in T_1\} - \#\{(a,a^{-1}): a\in T_1\} - \#\{(a,a): a\in T_1\} \right)\\
		&=\frac{1}{2}(k_1^2-2k_1) .
		\end{align*}
		For Case (b), we follow the same argument for Case (b) of $C_{0,2}$. 
	Consequently, 
	\begin{align*}
		C_{1,2}&=2|T_1^2\setminus (T_1^{(2)} \cup \{e\})|- \frac{|(T_1^2\setminus (T_1^{(2)} \cup \{e\}))\cap \hat{T}^{(4)}|}{2}\\
		&=k_1^2-2k_1-\frac{|(T_1^2\setminus (T_1^{(2)} \cup \{e\}))\cap \hat{T}^{(4)}|}{2}. \qedhere
	\end{align*} 
\end{proof}

By Lemma \ref{le:C_k,l}, it is straightforward to get
\begin{equation}\label{eq:aiTk_ajTk}
	\sum_{i<j} |a_i T_k\cap a_j T_k|=
	\begin{cases}
		2k_0(k_0-1)-\left|(T_0^2 \setminus T_0^{(2)})\cap \hat{T}^{(4)}\right|, & k=0;\\[0.5em]
		2k_1(k_1-1)-\frac{|T_1 \cap \hat{T}^{(2)}|}{2}-\left|(T_1^2\setminus (T_1^{(2)} \cup \{e\}))\cap \hat{T}^{(4)}\right|, &k=1.
	\end{cases}
\end{equation}
Recall that $\theta_0=|(T_0^2 \setminus T_0^{(2)})\cap \hat{T}^{(4)}|$. Let $\theta_{1,1}=|(T_1^2\setminus (T_1^{(2)} \cup \{e\}))\cap \hat{T}^{(4)}|$ and $\theta_{1,2}=|T_1 \cap \hat{T}^{(2)}|$ which equals $|T_1^{(2)} \cap \hat{T}^{(4)}|$. By \eqref{eq:T0^2+T1^2}
and $\hat{T}^{(4)} \cap T_0^{(2)} = \{e\}$, 
\begin{equation}\label{eq:sum_theta_i}
	\theta_0 +\theta_{1,1}+\theta_{1,2}=|\hat{T}^{(4)}\setminus\{e\}|=2n.
\end{equation}
Hence
\[
\theta_0+\theta_1 =\theta_0+ \theta_{1,1}+\frac{\theta_{1,2}}{2}\leq 2n.
\]

Now we are ready to prove Lemma \ref{le:in-ex_X_Y}.
\begin{proof}[Proof of Lemma \ref{le:in-ex_X_Y}.]
	We only prove the result for $T_0$. The proof of \eqref{eq:sum_|Y_i|-0} is the same.
	
	By the inclusion--exclusion principle, we count the distinct elements in $\hat{T}^{(2)}T_0$,
	\begin{equation}\label{eq:expand_T2T_0}
	|\hat{T}^{(2)}T_0| = \sum_{i=0}^{2n} |a_iT_0| - \sum_{i<j} |a_iT_0 \cap a_jT_0| + \sum_{r\geq 3} (-1)^{r-1} |a_{i_1}T_0\cap a_{i_2}T_0 \cap \cdots \cap a_{i_r}T_0|,
	\end{equation}
	where $i_1<i_2<\cdots <i_r$ cover all the possible values. By the definition of $X_i$'s, the left-hand side of \eqref{eq:expand_T2T_0} equals $\sum_{i=1}^{M_0}|X_i|$.
	
	Now we determine the value of the different sums in \eqref{eq:expand_T2T_0}. First, $\sum_{i=0}^{2n} |a_iT_0|=(2n+1)k_0$. Second,  \eqref{eq:aiTk_ajTk} tells us the value of $\sum_{i<j} |a_i T_0\cap a_j T_0|$. 
	
	Finally, suppose that $g\in a_{i_1}T_0\cap a_{i_2}T_0 \cap \cdots \cap a_{i_r}T_0$ with $r\geq 3$. It means that $g\in X_s$ for some $s\geq 3$. Then the contribution for $g$ in the sum $\sum_{r\geq 3} (-1)^{r-1} |a_{i_1}T_0\cap a_{i_2}T_0 \cap \cdots \cap a_{i_r}T_0|$ is
	\[ (-1)^{3-1} \binom{s}{3}+(-1)^{4-1} \binom{s}{4} +\cdots =\binom{s}{2} -\binom{s}{1} + \binom{s}{0} = \frac{(s-1)(s-2)}{2}.\]
	Therefore,
	\[\sum_{r\geq 3} (-1)^{r-1} |a_{i_1}T_0\cap a_{i_2}T_0 \cap \cdots \cap a_{i_r}T_0| = \sum_{s\geq 3} |X_s|\frac{(s-1)(s-2)}{2}.\]
	
	Plugging them all into \eqref{eq:expand_T2T_0}, we get \eqref{eq:sum_|X_i|-0}.
\end{proof}

Recall that at the very beginning of this section we have obtained
\begin{align}
	\label{eq:T_0^3}	\hat{T}^{(2)}T_0 &= (2k_0-k_1)H+T_1-T_0^3+2nT_0,\\
	\label{eq:T_1^3}	\hat{T}^{(2)}T_1 &= (2k_1-k_0)H+T_0-T_1^3+2nT_1.
\end{align}

In the following,  we investigate \eqref{eq:T_0^3} and \eqref{eq:T_1^3} separately in different cases depending on the value of $n$ modulo $3$.

\subsection{Case: $n\equiv 1\pmod{3}$}
First we need to prove the following  observation.

\begin{lemma}\label{le:n=1mod3_rep}
	When $n\equiv 1 \pmod{3}$, 
	\begin{enumerate}[label=(\alph*)]
		\item there is no element in $T_j^{(3)}$ appearing more than $2$ times for $j=0,1$,
		\item $e$ appears only once in $T_0^{(3)}$, and
		\item there are $0$ or $2$ elements appearing twice in $T_0^{(3)}$, and there are at most $2$ elements appearing twice in $T_1^{(3)}$.
	\end{enumerate}
\end{lemma}
\begin{proof}
	Recall that $|H|=n^2+n+1$. By the assumption, $3\mid |H|$ but $9\nmid |H|$. Thus there are exactly two elements in $H$ of order $3$. We denote one of them by $h$, which means the other one is $h^{-1}$.
	
	(a). For fixed $j=0$ or $1$, suppose that there exist distinct elements $a,b$ in $T_j$ such that $a^3=b^3$. Then $ab^{-1}$ must be of order $3$ in $H$. 
	Without loss of generality, assume $ab^{-1}=h$. 
	
	 Toward a contradiction, suppose that there exists $c\neq a,b$ in $T_j$ such that $c^3=a^3=b^3$ which means $ac^{-1}$ or $bc^{-1}$ equals $h$. By \eqref{eq:T0^2+T1^2}, there are at most one unordered pair $\{u,v\}$ with $u,v\in T_j$ such that $uv=h$. It follows that the unordered pairs $\{a,b^{-1}\}$ and $\{a,c^{-1} \}$ (or $\{b,c^{-1}\}$) must be equal. However, this is impossible. 
	
	(b). Note that $e=a^3$ for some $a\in T_0$ if and only if $a=e, h$ or $h^{-1}$. Hence, if $e$ appears more than once in $T_0^{(3)}$, then it appears for at least three times in $T_0^{(3)}$ which contradicts (a).
	
	(c). For $j=1$ or $2$, if $g=t_1^3=t_2^3$ appears twice in $T_j^{(3)}$, then $t_1t_2^{-1}$ has to be of order $3$. By \eqref{eq:T0^2+T1^2}, the unordered pair $\{t_1,t_2^{-1}\}$ must be $\{a,b^{-1} \}$ or $\{b,a^{-1} \}$ which are given in (a). Therefore, there is no element different from $a^3$ and $a^{-3}$ appearing twice in $T_j^{(3)}$. 
	
	In particular, for $j=0$, (b) tells us that $a^3\neq e$. Hence $a^3\neq a^{-3}\in T_0^{(3)}$. 
\end{proof}

	By Lemma \ref{le:T_0T_1_nece_condi} (e) and (f), $ 2k_0-k_1\equiv 2k_1-k_0\equiv 0\pmod{3} $. Using \eqref{eq:T_0^3}, \eqref{eq:T_1^3} and $3\mid (n-1)$, 
\begin{align}
\label{eq:1:T_0^3_mod3}	\hat{T}^{(2)}T_0 &\equiv T_1-T_0^{(3)} + 2T_0\pmod{3},\\
\label{eq:1:T_1^3_mod3}	\hat{T}^{(2)}T_1 &\equiv T_0-T_1^{(3)} + 2T_1\pmod{3}.
\end{align}
By Lemma \ref{le:T_0T_1_nece_condi} (b) and (h), 
\[T_0\cap T_1=\emptyset, \text{ and }T_0\cap T_0^{(3)} = \{e\}. \]
Let $\Delta_i$ be the set of elements appearing twice in $T_i^{(3)}$ for $i=1,2$. By Lemma \ref{le:n=1mod3_rep}, $|\Delta_0|, |\Delta_1|\leq 2$ and there is no element appearing more than 2 times in $T_i^{(3)}$ for $i=1,2$.  Note that 
\[
|T_0^{(3)}|=k_0-|\Delta_0|, \text{ and } |T_1^{(3)}|=k_1-|\Delta_1|.
\]
Depending on the parity of $n$, we investigate \eqref{eq:1:T_1^3_mod3} and \eqref{eq:1:T_0^3_mod3}, respectively.
\begin{theorem}\label{th:n=1_even}
	For $n\equiv 4 \pmod{6}$ and $n>2$, there is no inverse-closed subsets $T_0$ and $T_1\subseteq H$ with $e\in T_0$ and $k_0+k_1=2n+1$ satisfying \eqref{eq:T0T1} and \eqref{eq:T0^2+T1^2}.
\end{theorem}
\begin{proof}
	Suppose to the contrary that there exist $T_0$ and $T_1$ satisfying \eqref{eq:T0T1} and \eqref{eq:T0^2+T1^2}.
	
	As $2\mid n$,  $k_0=n+1$ and $k_1=n$. We concentrate on \eqref{eq:1:T_1^3_mod3}. Let $u_0=|T_1^{(3)} \cap T_0|$ and $u_1=|T_1^{(3)} \cap T_1|$. By Lemma \ref{le:T_0T_1_nece_condi}, $T_0\cap T_1=\emptyset$.  Furthermore, we separate $\Delta_1$ into three disjoint parts $\Delta_1^0 =\Delta_1\cap T_0$, $\Delta_1^1 =\Delta_1\cap T_1$ and $\Delta_1^2 =\Delta_1\setminus (T_0\cup T_1)$.
	
	By comparing the coefficients of elements in \eqref{eq:1:T_1^3_mod3}, we get
	\[
		\bigcup_{i\geq 0}Y_{3i+2} = (T_1\setminus T_1^{(3)})~\dot{\cup}~ \left(T_1^{(3)}\setminus (T_0~\dot{\cup}~T_1 \cup\Delta_1)\right)~\dot{\cup}~ \Delta_{1}^0.
	\]
	Note that $|T_1^{(3)}|=k_1-|\Delta_1|$.
	It follows that
	\[
	\sum_{i\geq 0}|Y_{3i+2}|=(n-u_1)+ (n-|\Delta_1|-(u_0+u_1)-|\Delta_1^2|)+|\Delta_1^0|.
	\]
	Similarly, 
	\[
	\bigcup_{i\geq 0}Y_{3i+1} = (T_0\setminus T_1^{(3)})~\dot{\cup}~\Delta_1^2~\dot{\cup}~(T_1\cap T_1^{(3)}\setminus \Delta_1^1).
	\]
	Thus
	\[
	\sum_{i\geq 0}|Y_{3i+1}|=(n+1-u_0)+|\Delta_1^2|+u_1-|\Delta_1^1|.
	\]
	To summarize, we have proved
	\begin{equation}\label{eq:1:Y_i_n=0}
	\sum_{i\geq 0}|Y_{3i+2}|=2n-2u_1-u_0-2|\Delta_1^2|-|\Delta_1^1|,~ \sum_{i\geq 0}|Y_{3i+1}| =n+1-u_0 + u_1+|\Delta_1^2|-|\Delta_1^1|, 
	\end{equation}
	Now \eqref{eq:sum_|Y_i|-0} becomes
	\begin{align}
\nonumber	\sum_{i=1}^{M_1}|Y_i|&=(2n+1)n-2(n-1)n +\theta_1 + \sum_{s\geq 3} \frac{(s-1)(s-2)}{2}|Y_s|\\
	&=3n +\theta_1+ \sum_{s\geq 3} \frac{(s-1)(s-2)}{2}|Y_s|. \label{eq:1:Y_i_n=0_extra}
	\end{align}
	Plugging  the two equations in \eqref{eq:1:Y_i_n=0} into it to replace $3n$, we get
	\begin{align*}
		\sum_{i=1}^{M_1}|Y_i|=&\sum_{i\geq 0}|Y_{3i+1}| + \sum_{i\geq 0}|Y_{3i+2}| + 2u_0+u_1+   |\Delta_1^2|+2|\Delta_1^1|-1\\
		&+\theta_1+ \sum_{s\geq 3} \frac{(s-1)(s-2)}{2}|Y_s|,
	\end{align*}
	which implies
	\begin{align}
	\nonumber	1=&2u_0+u_1+   |\Delta_1^2|+2|\Delta_1^1| +\theta_1 \\
	\label{eq:n=1_extra}		& +\sum_{s>3, 3\mid s} \left(\frac{(s-1)(s-2)}{2}-1\right)|Y_s| + \sum_{s\geq 3, 3\nmid s} \left(\frac{(s-1)(s-2)}{2}\right)|Y_s|. 
	\end{align}
	As $|Y_i|$, $u_0$, $u_1$,$|\Delta_1^2|$, $|\Delta_1^1|$ and $\theta $ are all nonnegative integers, \eqref{eq:n=1_extra} implies that $u_0=0$ and each $|Y_i|=0$ for $i\geq 4$. Recall that $u_1$ is the cardinality of the inverse-closed subset $T_1^{(3)}\cap T_1$ which does not contain the identity element $e$ of $H$. Hence $u_1$ must be even. Similarly, $|\Delta_1^1|$ and $|\Delta_1^2|$ are also even. Thus, by \eqref{eq:n=1_extra}, 
	\[u_1=|\Delta_1^1|=|\Delta_1^2|=0, \]
	and
	$\theta_1=1$. Plugging them into \eqref{eq:1:Y_i_n=0}, we get
	\[
	|Y_{1}|=n+1, ~ |Y_{2}|=2n.
	\]
	
	By \eqref{eq:sum_i|Y_i|}, 
	\[
	3|Y_3| = (2n+1)n-(n+1)-4n=2n^2-4n-1  \equiv 3 \pmod{6}.
	\]
	Consequently, $|Y_1|$ and $|Y_3|$ are both odd.
	As $Y_i$'s form a partition of the group $H$ and each $Y_i$ is inverse-closed, there exists only one of them of odd size. We get a contradiction.

	Therefore, we have excluded the existence of $|Y_i|$'s which means there do not exist inverse-closed subsets $T_0$ and $T_1\subseteq H$ with $e\in T_0$ and $k_0+k_1=2n+1$ satisfying \eqref{eq:T0T1} and \eqref{eq:T0^2+T1^2}.
\end{proof}

For $n$ odd and $n\equiv 1 \pmod{3}$, we  could not derive any contradiction by analyzing the value of $|X_i|$'s and $|Y_i|$'s. 

\subsection{Case: $n\equiv 0 \pmod{3}$}
We separate the proof into two cases according to the parity of $n$.
\begin{theorem}\label{th:n=0_odd}
	For $n\equiv 3 \pmod{6}$, there is no inverse-closed subsets $T_0$ and $T_1\subseteq H$ with $e\in T_0$ and $k_0+k_1=2n+1$ satisfying \eqref{eq:T0T1} and \eqref{eq:T0^2+T1^2}.
\end{theorem}
\begin{proof}
	As $n$ is odd, $|T_0|=k_0=n$ and $|T_1|=k_1=n+1$. 	
	Toward a contradiction, suppose that $T_0$ and $T_1$ exist such that \eqref{eq:T0T1} and \eqref{eq:T0^2+T1^2} hold.	
	
	First Lemma \ref{le:T_0T_1_nece_condi} tells us that $ 2k_0-k_1\equiv 2k_1-k_0\equiv 2\pmod{3} $. Under the assumption that $3\mid n$, 
	\begin{align}
	\label{eq:T_0^3_mod3}	\hat{T}^{(2)}T_0 &\equiv 2H+T_1-T_0^{(3)}\pmod{3},\\
	\nonumber	\hat{T}^{(2)}T_1 &\equiv 2H+T_0-T_1^{(3)}\pmod{3}.
	\end{align}
	Moreover, as $3\nmid |H|$, there is no repeating element in $T_0^{(3)}$ and $T_1^{(3)}$.
	
	Assume that there are $\ell_0$ elements in $T_1\cap T_0^{(3)}$ and $\ell_1$ elements in $T_0\cap T_1^{(3)}$. It follows that
	\begin{align*}
	\sum_{i\geq 0}|X_{3i+1}|=k_0-\ell_0,  \quad & \sum_{i\geq 0}|X_{3i+2}| = n^2-n+2\ell_0, & \sum_{i\geq 0}|X_{3i}| = k_1-\ell_0, \\
	\sum_{i\geq 0}|Y_{3i+1}|=k_1-\ell_1, \quad  & \sum_{i\geq 0}|Y_{3i+2}| = n^2-n+2\ell_1, & \sum_{i\geq 0}|Y_{3i}| = k_0-\ell_1. 
	\end{align*}
	In particular, for $2\nmid n$, 
	\begin{equation}
	\label{eq:X_i_n=0}	\sum_{i\geq 0}|X_{3i+1}|=n-\ell_0,  \quad  \sum_{i\geq 0}|X_{3i+2}| = n^2-n+2\ell_0, \quad  \sum_{i\geq 0}|X_{3i}| = n+1-\ell_0. \\
	\end{equation}
	
	The strategy of our proof is as follows: First we consider the possible nonnegative integer solutions for $|X_i|$'s satisfying \eqref{eq:X_i_n=0}. Several inequalities will provide us
	\begin{equation}\label{eq:X_0_size}
		\begin{cases}
		|X_0|=n+1, \\ 
		\ell_0=0,
		\end{cases}
		  \text{ or } \quad 
  		\begin{cases}
		  |X_0|=n-1, \\ 
		  \ell_0=0 \text{ or }2.
		  \end{cases}
	\end{equation}
	Then, we will show that such $X_0$ cannot exist.
	 
	 \medskip
	 Next we provide the details of the proof of \eqref{eq:X_0_size}.
	Now \eqref{eq:sum_|X_i|-0} becomes 
	\[\sum_{i=1}^{M_0}|X_i|=3n+\theta_0+ \sum_{s\geq 3} \frac{(s-1)(s-2)}{2}|X_s|.\]	
	Subtract \eqref{eq:sum_|X_i|} from \eqref{eq:sum_|X_i|-0}, 
	\[
	-|X_0| = -n^2+2n-1+\theta_0+\sum_{s\geq 3} \frac{(s-1)(s-2)}{2}|X_s|,
	\]
	which means
	\begin{equation}\label{eq:Gamma_0}
		n^2-|X_0|-2n-\theta_0+1 = \sum_{s\geq 3} \frac{(s-1)(s-2)}{2}|X_s|\leq \frac{M_0-1}{2}\sum_{s\geq 3} (s-2)|X_s|.
	\end{equation}
	On the other hand, by \eqref{eq:sum_i|X_i|}, \eqref{eq:sum_|X_i|} and \eqref{eq:X_i_n=0},  
	\begin{align}
	\nonumber	&\sum_{i\geq 1}3i(|X_{3i+1}| + |X_{3i+2}| ) + \sum_{i\geq 1} (3i-2)|X_{3i}|\\ 
	\label{eq:Gamma1}	=& \sum_{i\geq 1} i|X_i| - \sum_{i\geq 0}|X_{3i+1}| - 2\sum_{i\geq 0}|X_{3i+2}|-\sum_{i\geq 0} 2|X_{3i}| +2|X_0|\\
	\nonumber	=& 2|X_0|-\ell_0-2.
	\end{align}
	As $M_0>1$, by Lemma \ref{le:e_in_X_1}, $|X_{M_0}|$ is even, and $|X_{M_0}|\geq 2$. By \eqref{eq:Gamma1},
	\[  (M_0-2)2\leq   2|X_0|-\ell_0-2,\]
	which means
	\begin{equation}\label{eq:M_0<=}
		M_0 \leq |X_0|+1-\frac{\ell_0}{2}.
	\end{equation}
	Moreover, by \eqref{eq:Gamma1}, we also have
	\[
	\sum_{s\geq 3}(s-2)|X_s| \leq2|X_0|-\ell_0-2.
	\]
	Plugging it and \eqref{eq:M_0<=} into \eqref{eq:Gamma_0}, we have
	\begin{align}
	\nonumber		n^2-|X_0|-2n-\theta_0+1 &\leq \frac{M_0-1}{2}(2|X_0|-\ell_0-2)\\
	\label{eq:Gamma_4}	&\leq (|X_0|-\ell_0/2)(|X_0|-\ell_0/2-1)\\
	\nonumber	&\leq |X_0|(|X_0|-1).
	\end{align}
	By $|X_0|\leq \sum_{i\geq0}|X_{3i}|=k_1-\ell_0\leq n+1$ and a direct computation with the above inequality using $\theta_0\leq 2n$, we derive that $|X_0|=n+1, n$, $n-1$ or $n-2$. Moreover, $|X_1|$ is odd by Lemma \ref{le:e_in_X_1}. Consequently, by \eqref{eq:X_i_n=0}, $\ell_0$ must be even for odd $n$. Taking \eqref{eq:Gamma_4} and $\sum_{i\geq 0}|X_{3i}| = n+1-\ell_0$ into account, we get four possible cases
	\begin{enumerate}[label=(\Roman*)]
		\item $|X_0|=n+1$ and $\ell_0=0$;
		\item $|X_0|=n-1$ and $\ell_0=0$ or $2$;
		\item $|X_0|=n$;
		\item $|X_0|=n-2$.
	\end{enumerate}
	By Lemma \ref{le:e_in_X_1}, $|X_0|$ must be even. Consequently Case (III) and Case (IV) are impossible, because $n$ is assumed to be odd.
	
	Hence, we have proved \eqref{eq:X_0_size}.
	\medskip

	\textbf{Determination of $X_0$}. 	Next we show that
	\begin{equation}\label{eq:determining_X_0}
		X_0=\begin{cases}
		T_1, & |X_0|=n+1;\\
		T_1\setminus \{\gamma,\gamma^{-1}\}, & |X_0|=n-1,
		\end{cases}
	\end{equation}
	where $\gamma$ is an element in $T_1$. Afterwards we will finish the whole proof by getting some contradiction from \eqref{eq:determining_X_0}.
	
	Recall
	\begin{equation}
	\hat{T}^{(2)}T_0 = (n-1)H+T_1-T_0^3+2nT_0. \tag{\ref{eq:T_0^3}}
	\end{equation}
	
	First we consider the case with $\ell_0=0$. For an arbitrary element $t\in T_1$, since $\ell_0=|T_1\cap T_0^{(3)}|=0$, $t$ can only be represented as
	\[ abc \text{ for some pairwise distinct }a,b,c\in T_0,\quad \text{or} \quad  aab   \text{ for some distinct }a,b\in T_0,\]
	in $T_0^3$.
	
	Taking account of the ordering of the elements in the representations, the coefficient of $t$ in $T_0^3\in \Z[H]$ is divisible by $3$.  Together with $n\equiv 0\pmod{3}$ and $T_1\cap T_0=\emptyset$, the coefficient of $t$ in the right-hand side of \eqref{eq:T_0^3} must be congruent to $0$ modulo $3$.  Since $\sum_{i\geq 0} |X_{3i}|=n+1$ and $|T_1|=n+1$, 
	\begin{equation*}
		X_0=
		\begin{cases}
		T_1, & |X_0|=n+1;\\
		T_1\setminus\{\gamma,\gamma^{-1}\}, & |X_0|=n-1,
		\end{cases}
	\end{equation*}
	 for some $\gamma\in T_1$.
	
	When $\ell_0=2$, by \eqref{eq:X_0_size}, $|X_0|=n-1$.  By our definition of $\ell_0$,  there exist exactly two elements $\gamma$ and $\gamma^{-1}\in T_1$ such that $\gamma,\gamma^{-1}\in T_0^{(3)}$. The coefficients of them in the right-hand side of \eqref{eq:T_0^3} must be congruent to $1$ modulo $3$, and the coefficient of any element in $T_1\setminus\{\gamma,\gamma^{-1}\}$ is still congruent to $0$ modulo $3$. Thus $X_0=T_1\setminus\{\gamma,\gamma^{-1}\}$.

	Therefore, we have obtained \eqref{eq:determining_X_0}.
	
	In the final step, we show that \eqref{eq:determining_X_0} cannot hold which concludes the proof of Theorem \ref{th:n=0_odd}.
	
	If $T_1=X_0$, then \eqref{eq:T_0^3} tells us
	\[
	(T_0^{(2)}+T_1^{(2)})T_0\cap T_1=\emptyset,
	\]
	which cannot hold, because $T_0T_1=H-e$ means for any non-identity $a\in T_0^{(2)}+T_1^{(2)}$ there always exist $t_0\in T_0$ and $t_1\in T_1$ such that $t_0t_1=a$.
	
	If $T_1\setminus\{\gamma, \gamma ^{-1}\}=X_0$, then \eqref{eq:T_0^3} tells us
	$t\notin (T_0^{(2)}+T_1^{(2)})T_0$ for any $t\in T_1\setminus\{\gamma, \gamma^{-1}\}$. By \eqref{eq:T0T1}, 
	\[ (T_1-\gamma -\gamma^{-1})  T_0 = H-e-(\gamma T_0 + \gamma^{-1}T_0). \]
	For any $x\in T_0$, if $tx=y$, then $t=x^{-1}y$ which means $y\notin T_0^{(2)}+T_1^{(2)}$. 
	Thus the above equation implies
	\begin{equation}\label{eq:gammaT_0+gammaT_1}
		\gamma T_0 + \gamma^{-1}T_0 +e= T_0^{(2)}+T_1^{(2)}.
	\end{equation}
	As $\gamma^2\in T_1^{(2)}$, 
	\[ \gamma^2=\gamma a \quad \text{or} \quad  \gamma^2 =\gamma^{-1}a \]
	for some $a\in T_0\setminus \{e\}$. The first case is impossible because $T_0\cap T_1=\emptyset$.  The second case implies that $a=\gamma^3\in T_0\cap T_1^{(3)}$. 
	By \eqref{eq:gammaT_0+gammaT_1}, $\gamma a=\gamma^4$ belongs to $T_0^{(2)}$ or $T_1^{(2)}$. 
	
	If $\gamma^4\in T_0^{(2)}$, then $\gamma^2\in T_0$ which means $T_0 \cap T_1^{(2)}\neq \emptyset$. It contradicts Lemma \ref{le:T_0T_1_nece_condi} (c). 
	
	For $\gamma^4\in T_1^{(2)}$, we need more detailed analysis. As $a^2\in T_0^{(2)}$, by \eqref{eq:gammaT_0+gammaT_1} there exists $b\in T_0$ such that 
	\[ a^2=\gamma^6 =  \gamma b, \text{ or } \gamma^{-1}b, \]
	 which means $b=\gamma^5$ or $\gamma^7\in T_0$. If $b=\gamma^5\in T_0$, then 
	 \[T_1^{(2)}\ni \gamma^2=\gamma^5 \cdot \gamma^{-3}=b\cdot a^{-1}\in T_0^2 \] 
	 whence $\gamma^2\in T_0^2 \cap T_1^{(2)}$ which is not possible by Lemma \ref{le:T_0T_1_nece_condi} (g); if $b=\gamma^7\in T_0$, then $\gamma^4=b \gamma^{-3}=ba^{-1}\in T_0^2$. Recall that $\gamma^4\in T_1^{(2)}$. Hence $\gamma^4\in T_0^2\cap T_1^{(2)}$ which contradicts Lemma \ref{le:T_0T_1_nece_condi} (g) again.

	To summarize, we have proved that there exist no $T_0$ and $T_1$ satisfying \eqref{eq:T0T1} and \eqref{eq:T0^2+T1^2}. 
\end{proof}

The proof for even $n$ which is congruent to $0$ modulo $3$ is similar, but we have to handle some extra problems.
\begin{theorem}\label{th:n=0_even}
	For $n\equiv 0 \pmod{6}$, there is no inverse-closed subsets $T_0$ and $T_1\subseteq H$ with $e\in T_0$ and $k_0+k_1=2n+1$ satisfying \eqref{eq:T0T1} and \eqref{eq:T0^2+T1^2}.
\end{theorem}
\begin{proof}
	Now $n$ is divisible by $2$ and $3$, $|T_0|=k_0=n+1$ and $|T_1|=k_1=n$. 	
	Toward a contradiction, suppose that $T_0$ and $T_1$ exist such that \eqref{eq:T0T1} and \eqref{eq:T0^2+T1^2} hold.	
	The smallest possible value of $n$ equals $6$. However, throughout the rest part of the proof, we can always assume that 
	\begin{equation}\label{eq:n=0_even_n>=30}
	n\geq 30.
	\end{equation}
	The reason is as follows: for $n=6,12$ and $24$, $n^2+n+1$ is prime which means the nonexistence result is already covered by Result \ref{result:prime}. For $n=18$, it can be excluded by Proposition \ref{prop:known}.
	
	Instead of looking at $X_i$'s as in the proof of Theorem \ref{th:n=0_odd}, we consider $Y_i$'s. By \eqref{eq:T_1^3},
	\begin{equation}\label{eq:T_1^3_mod3}
	\hat{T}^{(2)}T_1 \equiv 2H+T_0-T_1^{(3)}\pmod{3}.
	\end{equation}
	It follows that
	\begin{align}\label{eq:0:even:Y_i_n=0}
		\sum_{i\geq 0}|Y_{3i+1}|=n-\ell_1,  \quad & \sum_{i\geq 0}|Y_{3i+2}| = n^2-n+2\ell_1, & \sum_{i\geq 0}|Y_{3i}| = n+1-\ell_1,
	\end{align}
	where $\ell_1$ is defined to be the cardinality of $T_0\cap T_1^{(3)}$, 
	and 	\eqref{eq:sum_|Y_i|-0} becomes 
	\[\sum_{i=1}^{M_1}|Y_i|=3n +\theta_1 + \sum_{s\geq 3} \frac{(s-1)(s-2)}{2}|Y_s|.\]
	By subtracting \eqref{eq:sum_|X_i|} from \eqref{eq:sum_|Y_i|-0},	we get
	\begin{equation}\label{eq:0:even:Gamma_0}
		n^2-|Y_0|-2n+1 -\theta_1= \sum_{s\geq 3} \frac{(s-1)(s-2)}{2}|Y_s|.
	\end{equation}
	On the other hand, by \eqref{eq:sum_i|Y_i|}, \eqref{eq:sum_|X_i|} and \eqref{eq:0:even:Y_i_n=0},  
	\begin{align}
		\nonumber	&\sum_{i\geq 1}3i(|Y_{3i+1}| + |Y_{3i+2}| ) + \sum_{i\geq 1} (3i-2)|Y_{3i}|\\ 
		\label{eq:0:even:Gamma1}	=& \sum_{i\geq 1} i|Y_i| - \sum_{i\geq 0}|Y_{3i+1}| - 2\sum_{i\geq 0}|Y_{3i+2}|-\sum_{i\geq 0} 2|Y_{3i}| +2|Y_0|\\
		\nonumber	=& 2|Y_0|-\ell_1-2.
	\end{align}
	It is easy to check that \eqref{eq:0:even:Gamma1} is the same as \eqref{eq:Gamma1} by replacing $X_i$ with $Y_i$, and $\ell_0$ with $\ell_1$.
	
	As we do not know whether $|Y_{M_1}|\geq 2$, we cannot immediately follow the proof of Theorem \ref{th:n=0_odd} to determine the possible value of $|Y_i|$'s. Next we concentrate on  the proof of the following claim.
	
	\medskip
	\textbf{Claim 1.} $|Y_{M_1}|\geq 2$.
	\medskip
	
	Toward a contradiction, suppose that $|Y_{M_1}|=1$ which is equivalent to $Y_{M_1}=\{e\}$. By 
	\[
	\hat{T}^{(2)}T_1 =\sum_{i=1}^{M_1} i Y_i,
	\]
	we know that $M_1=|\hat{T}^{(2)}\cap T_1|=|T_0^{(2)}\cap T_1| + |T_1^{(2)}\cap T_1|\leq k_1=n$, and $2\mid M_1$ because $|T_0^{(2)}\cap T_1|$ and $|T_1^{(2)}\cap T_1|$ are both even.	Moreover, by \eqref{eq:T_1^3_mod3}, $e\in T_0$ and $e\notin T_1^{(3)}$, we see that $3\mid M_1$. Therefore, we have obtained the following restriction on $M_1$:
	\begin{equation}\label{eq:M1_res}
		6\leq M_1\leq n, \text{ and }6\mid M_1.
	\end{equation}
	
	By \eqref{eq:0:even:Gamma1}, we obtain $M_1-2\leq 2|Y_0|-\ell_1-2$, which means
	\begin{equation}\label{eq:M1_upper_Y0}
		M_1\leq 2|Y_0|-\ell_1.
	\end{equation}
	
	Let $M_1' = \max \{i<M_1: |Y_i|>0\}$. If $M'_1<3$, then by \eqref{eq:0:even:Y_i_n=0}, \eqref{eq:M1_res} and $|Y_{M_1}|=1$, we have
	\begin{align*}
			|Y_0|&=n-\ell_1,\\
			|Y_1|&=n-\ell_1,\\
			|Y_2|&=n^2-n+2\ell_1.
	\end{align*}
	Plugging them in to \eqref{eq:0:even:Gamma_0}, we get
	\[
	n^2-3n+\ell_1+1 -\theta_1=\frac{(M_1-1)(M_1-2)}{2}\leq \frac{n^2-3n+2}{2},
	\]
	where the last inequality comes from \eqref{eq:M1_res}. Recall that $\theta_1\leq 2n$. Thus 
	\[
		n^2-3n+\ell_1+1 - 2n\leq  \frac{n^2-3n+2}{2}.
	\]
	It implies that $n\leq 7$ which contradicts our assumption \eqref{eq:n=0_even_n>=30}. Therefore $M'_1\geq 3$.
	
	As $e\notin Y_{M'_1}$, by the inverse-closed property of $\hat{T}^{(2)}$ and $T_1$, $|Y_{M'_1}|$ must be even. In particular,  $|Y_{M'_1}|\geq 2$. By \eqref{eq:0:even:Gamma1}, we get
	\[(M_1'-2)\cdot 2+(M_1-2) \leq 2|Y_0| -\ell_1-2, \]
	which means
	\begin{equation}\label{eq:M1'_upper}
		M'_1\leq |Y_0|-\frac{\ell_1+M_1}{2}+2.
	\end{equation}
	By \eqref{eq:0:even:Gamma1}, we also get
	\begin{equation}\label{eq:Y0_further}
		2|Y_0| -\ell_1-2\geq \sum_{3\leq s\leq M'_1} (s-2)|Y_s|+(M_1-2).
	\end{equation}
	By \eqref{eq:0:even:Gamma_0}, 
	\[
		n^2-|Y_0|-2n+1 -\theta_1\leq \frac{M'_1-1}{2} \sum_{3\leq s\leq M'_1} (s-2)|Y_s| + \frac{(M_1-1)(M_1-2)}{2}.
	\]
	Plugging \eqref{eq:Y0_further} and \eqref{eq:M1'_upper} into it, we get
	\begin{align*}
	n^2-|Y_0|-2n+1 -\theta_1&\leq \frac{M'_1-1}{2} (2|Y_0|-\ell_1-M_1) + \frac{(M_1-1)(M_1-2)}{2}\\
		&\leq \left(
		|Y_0|-\frac{M_1+\ell_1}{2}+1
		\right)\left(
		|Y_0|-\frac{M_1+\ell_1}{2}
		\right)+ \frac{(M_1-1)(M_1-2)}{2}  
		\\
		&\leq \left(
		|Y_0|-\frac{M_1}{2}+1
		\right)\left(
		|Y_0|-\frac{M_1}{2}
		\right)+ \frac{(M_1-1)(M_1-2)}{2}.
	\end{align*}
	We add $|Y_0|$ on both sides of the above inequality and get
	\begin{equation}\label{eq:YM1>1_long}
		n^2-2n+1-\theta_1 \leq F(|Y_0|, M_1),
	\end{equation}
	  where
	\[
		F(x,y) := x^2+\frac{3}{4}y^2-xy+2x-2y+1.
	\]
	Recall that $|Y_0|\leq n-\ell_1$ by the third equation in \eqref{eq:0:even:Y_i_n=0} and \eqref{eq:M1_res}, we only have to consider the value of $F(x,y)$ for $(x,y)\in [0,n]\times [6,n]$. Our goal is to show that the value of $F(x,y)$ is always smaller than $n^2-2n+1-\theta_1$ which means \eqref{eq:YM1>1_long} cannot hold.
	
	It is easy to see that  $F$ defines an elliptic paraboloid and the minimum value is attained at point $(-1/2, 1)$. 
	 Hence, the maximum value of $F(x,y)$ for $(x,y)\in [0,n]\times [6,n]$ can be obtained only if $x=n$ or $y=n$.
	 
	 For $x=n$, 
	 \[
	 F(n,y) = \frac{3}{4}y^2-(n+2)y+n^2+2n+1.
	 \]
	 Its maximum value for $y\in [6,n]$ with $6\mid y$ is 
	 \[F(n,6) =n^2-4n+16,\]
	 and $F(n,12)= n^2-10n + 85$ is the second largest value under the assumption \eqref{eq:n=0_even_n>=30}.
	 
	 First, we can simply check that $y=6$ is impossible: Now $M_1=6$ and $|Y_0|=n$, which means $\ell_1=0$ and $|Y_3|=0$.  By \eqref{eq:0:even:Y_i_n=0}, \eqref{eq:sum_|Y_i|-0} and \eqref{eq:0:even:Gamma_0}, 
	 \begin{align*}
	 	|Y_1|+|Y_4|&=n,\\
	 	|Y_2|+|Y_5|&=n^2-n,\\
	 	|Y_1|+2|Y_2|+4|Y_4|+5|Y_5|&=2n^2+n-6,\\
	 	3|Y_4|+6|Y_5|&=n^2-3n+1-\theta_1-10.	 	
	 \end{align*}
	 Consequently, $|Y_1|=\frac{n^2}{3}-\frac{4n}{3}-\frac{\theta_1}{3}+1$ which must be smaller than or equal to $n$ by the first equation. Recall that $\theta_1\leq 2n$. Thus
	 \[
	 \frac{n^2}{3}-{2n}+1\leq n,
	 \]
	 which implies $n<9$. It contradicts our assumption \eqref{eq:n=0_even_n>=30}.
	 
	For the second largest value $F(n,12)$, 
	as $\theta_1\leq 2n$, if \eqref{eq:YM1>1_long} holds, then
	\[
	n^2-4n+1\leq n^2-2n+1 -\theta_1\leq n^2-10n + 85,
	\]
	which means $n\leq 14$. But these value have already been excluded by \eqref{eq:n=0_even_n>=30}. Therefore, 
	\[
	n^2-2n+1-\theta_1 > F(n,y),
	\]
	for any $y\in [6,n]$ with $6\mid y$.
	 
	 For $y=n$, 
	 \[
	 F(x,n) = x^2 - (n-2)x+\frac{3}{4}n^2-2n+1.
	 \]
	Its maximum value for $x\in [0,n]$ is 
	\[F(n,n) =\frac{3}{4}n^2+1.  \]
	It is larger than or equal to $n^2-4n+1$ if and only if $n\leq 16$, which again contradicts  \eqref{eq:n=0_even_n>=30}.
	
	Therefore, \eqref{eq:YM1>1_long} cannot hold which means that we have proved \textbf{Claim 1}.
	\medskip
	
	Following the proof of Theorem \ref{th:n=0_odd}, our next goal is to show that the only possible value of $|Y_0|$ are as follows:
	\begin{equation}\label{eq:Y_0_size}
		\begin{cases}
		|Y_0|=n+1, \\ 
		\ell_1=0,
		\end{cases}
		\text{ or } \quad 
		\begin{cases}
		|Y_0|=n-1, \\ 
		\ell_1=0 \text{ or }2,
		\end{cases}
		\text{ or } \quad 
		\begin{cases}
		|Y_0|=n, \\ 
		\ell_1=0.
		\end{cases}
	\end{equation}

	By replacing $Y_i$'s with $X_i$'s, one can get \eqref{eq:X_i_n=0} and \eqref{eq:Gamma1} from \eqref{eq:0:even:Y_i_n=0} and \eqref{eq:0:even:Gamma1}, respectively. 
	Moreover, \eqref{eq:Gamma_0}  becomes \eqref{eq:0:even:Gamma_0} if we change $X_i$'s to $Y_i$'s and change $\theta_0$ to $\theta_1$.
	
	Hence, by the same argument for \eqref{eq:Gamma_4}, we get
	\begin{align}
	\nonumber		n^2-|Y_0|-2n+1-\theta_1 &\leq \frac{M_1-1}{2}(2|Y_0|-\ell_1-2)\\
	\label{eq:even:Gamma_4}	&\leq (|Y_0|-\ell_1/2)(|Y_0|-\ell_1/2-1)\\
	\nonumber	&\leq |Y_0|(|Y_0|-1).
	\end{align}
	Taking account of $\theta_1\leq 2n$, we get $|Y_0|=n-1, n$ or $n+1$ under the assumption \eqref{eq:n=0_even_n>=30}. Note that $\ell_1=|T_0\cap T_1^{(3)}|$ must be even, because $T_0$ and $T_1$ are inverse-closed and $e\notin T_1^{(3)}$. Taking \eqref{eq:even:Gamma_4} and $|Y_0|\leq \sum_{i\geq 0} |Y_{3i}|=n+1-\ell_1$ into account, we get four possible cases
	\begin{enumerate}[label=(\Roman*)]
		\item $|Y_0|=n+1$ and $\ell_1=0$;
		\item $|Y_0|=n-1$ and $\ell_1=0$ or $2$;
		\item $|Y_0|=n$, $\ell_1=0$;
		\item $|Y_0|=n-2$, $\ell_1=0$.
	\end{enumerate}
	To prove \eqref{eq:Y_0_size}, we only have to exclude Case (IV).

	\textbf{Case (IV).} 
	By \eqref{eq:0:even:Gamma1} and \textbf{Claim 1}, 
	\[
	(M_1-2)2\leq 2|Y_0|-\ell_1-2,
	\]
	which implies 
	\begin{equation*}
		M_1\leq |Y_0|+1-\frac{\ell_0}{2}=n-2+1=n-1.
	\end{equation*}
	On the other hand, by \eqref{eq:even:Gamma_4}, 
	\[
	M_1\geq 1+ \frac{n^2-(n-2)-2n+1-\theta_1}{(n-2)-1}\geq n-1.
	\]
	Therefore $M_1$ must be $n-1$.  Plugging $\ell_1=0$ into the last equation of \eqref{eq:0:even:Y_i_n=0}, we get
	\begin{equation}\label{eq:n:even_IV}
		\sum_{i\geq 0}|Y_{3i}| = n+1.
	\end{equation}
	By \eqref{eq:0:even:Gamma1} and $M_1=n-1 \equiv 2 \mod 3$,
	\[
	2n-6=2|Y_0|-\ell_1-2=\sum_{i\geq 1}3i|Y_{3i+1}| + \sum_{i\geq 1, 3i+2\neq M_1}3i|Y_{3i+2}| + 2(n-1-2)+\sum_{i\geq 1} (3i-2)|Y_{3i}|.
	\]
	As a consequence, $|Y_3|=|Y_4|=\cdots=|Y_{M_1-1}|=0$. However, it means $\sum_{i\geq 0} |Y_{3i}|=|Y_0|=n-2$ which contradicts \eqref{eq:n:even_IV}.
	
	Therefore, we have proved \eqref{eq:Y_0_size}.
	\medskip
	
	\textbf{Determination of $Y_0$}. We basically follow the argument used in the proof of Theorem \ref{th:n=0_odd}. Our goal is to prove
		\begin{equation}\label{eq:determining_Y_0}
	Y_0=\begin{cases}
	T_0, &\text{ if }|Y_0|=n+1;\\
	T_0\setminus\{e\}, &\text{ if } |Y_0|=n;\\
	T_0\setminus \{\gamma,\gamma^{-1}\}, &\text{ if } |Y_0|=n-1,
	\end{cases}
	\end{equation}
	for some $\gamma\in T_0$.
	
	Recall
	\begin{equation}
	\hat{T}^{(2)}T_1 = (n-1)H+T_0-T_1^3+2nT_1. \tag{\ref{eq:T_1^3}}
	\end{equation}
	
	First let us consider the case with $\ell_1=0$. For an arbitrary element $t\in T_0$, as $\ell_1=|T_0\cap T_1^{(3)}|=0$, $t$ can only be represented as
	\[ abc \text{ for some pairwise distinct }a,b,c\in T_1,\quad \text{or} \quad  aab   \text{ for some distinct }a,b\in T_1,\]
		in $T_1^3$.
	
	Taking account of the ordering of the elements in the representations, the coefficient of $t$ in $T_1^3\in \Z[H]$ is divisible by $3$.  Together with $n\equiv 0\pmod{3}$ and $T_1\cap T_0=\emptyset$, the coefficient of $t\in T_0$ in the right-hand side of \eqref{eq:T_1^3} must be congruent to $0$ modulo $3$. 	
	Since $\sum_{i\geq 0} |Y_{3i}|=n+1$, $|T_0|=n+1$ and $Y_0$ is inverse-closed, we obtain \eqref{eq:determining_Y_0}.
	
	When $\ell_1=2$ which happens only if $|Y_0|=n-1$ by \eqref{eq:Y_0_size},  there exist exactly two elements $\gamma, \gamma^{-1}\in T_0\cap  T_1^{(3)}$. The coefficients of $\gamma$ or $\gamma^{-1}$ in the right-hand side of \eqref{eq:T_1^3} must be congruent to $1$ modulo $3$, and the coefficient of any element in $T_0\setminus\{\gamma,\gamma^{-1}\}$ is still congruent to $0$ modulo $3$. Thus $Y_0=T_0\setminus\{\gamma,\gamma^{-1}\}$ which corresponds to the last case in \eqref{eq:determining_Y_0}.
	\medskip

	The final part of our proof is to show that \eqref{eq:determining_Y_0} cannot hold which concludes the proof of Theorem \ref{th:n=0_even}.
	
	If $T_0=Y_0$, then \eqref{eq:T_1^3} tells us
	\[
	(T_0^{(2)}+T_1^{(2)})T_1\cap T_0=\emptyset,
	\]
	which cannot hold, because $T_0T_1=H-e$ means for any non-identity $a\in T_0^{(2)}+T_1^{(2)}$ there always exist $t_0\in T_0$ and $t_1\in T_1$ such that $t_0t_1=a$.
	
	If $T_0\setminus\{ e\}=Y_0$, then $t\notin (T_0^{(2)}+T_1^{(2)})T_1$ for any $t\in T_0\setminus\{e\}$. By \eqref{eq:T0T1}, 
	\[ Y_0T_1=(T_0-e)  T_1 = H-e-T_1. \]
	As $Y_0\cap(T_0^{(2)}+T_1^{(2)})T_1=\emptyset$ implies $Y_0T_1 \cap (T_0^{(2)}+T_1^{(2)})=\emptyset$, 
	\[T_0^{(2)}+T_1^{(2)}\subseteq e+T_1,\]
	which is impossible by the sizes of them.
	
	If $T_0\setminus\{ \gamma, \gamma^{-1}\}=Y_0$ for some $\gamma\neq e$, then $t\notin (T_0^{(2)}+T_1^{(2)})T_1$ for any $t\in T_0\setminus\{\gamma, \gamma^{-1}\}$. By \eqref{eq:T0T1}, 
	\[ Y_0T_1=(T_0-\gamma -\gamma^{-1})  T_1 = H-e-(\gamma T_1 + \gamma^{-1}T_1). \]
	As $Y_0\cap(T_0^{(2)}+T_1^{(2)})T_1=\emptyset$ implies $Y_0T_1 \cap (T_0^{(2)}+T_1^{(2)})=\emptyset $,  the above equation and the sizes of $T_0$ and $T_1$ lead to
	\begin{equation}\label{eq:cT_0+cT_1}
			\gamma T_1 + \gamma^{-1}T_1+e=T_0^{(2)}+T_1^{(2)}.
	\end{equation}
		
	Since $\gamma^2\in T_0^{(2)}$, 
	\[ \gamma^2=\gamma \alpha \quad \text{or} \quad  \gamma^2 =\gamma^{-1}\alpha \]
	for some $\alpha\in T_1$. The first case is impossible because $T_0\cap T_1=\emptyset$.  The second case implies that $\alpha=\gamma^3\in T_1\cap T_0^{(3)}$. 
	By \eqref{eq:cT_0+cT_1}, $\gamma \alpha=\gamma^4$ belongs to $T_0^{(2)}$ or $T_1^{(2)}$. 
	
	If $\gamma^4\in T_1^{(2)}$, then $\gamma^2\in T_1$ which means $\gamma^2\in T_1 \cap T_0^{(2)}$. It follows that the coefficient of $e$ in $\hat{T}^{(2)}T_1$ is larger than $0$. However, $e\in Y_0=T_0\setminus \{\gamma,\gamma^{-1}\}$. This is a contradiction.
	
	If $\gamma^4\in T_0^{(2)}$, then $\gamma^2\in T_0\cap T_0^{(2)}$.  It contradicts Lemma \ref{le:T_0T_1_nece_condi} (c).
	
	Therefore, we have finished the proof that there exist no $T_0$ and $T_1$ satisfying \eqref{eq:T0T1} and \eqref{eq:T0^2+T1^2}. 	
\end{proof}
The main result, i.e.\ Theorem \ref{th:main} is a simple combination of Theorems \ref{th:n=1_even}, \ref{th:n=0_odd} and \ref{th:n=0_even}.

\section{Concluding remarks}\label{sec:concluding}
In this paper, we investigate the classification of linear Lee codes 
of packing radius $2$ and packing density $\frac{|S(n,2)|}{|S(n,2)|+1}$ in $\Z^n$. 
We call linear Lee codes with this density almost perfect. There are still 
several cases remaining open.
\begin{problem}\label{prob:classification}
	For $n\equiv 1,2,5\pmod{6}$, show that there is an almost perfect linear 
	Lee code of packing radius $2$ if and only if $n=2$.	
\end{problem}
If this problem is solved, then we get the complete nonexistence proof of APLL 
codes of packing radius $2$, for $n\geq 3$. As a direct consequence, Example \ref{ex:G_n=1,2} provides all possible APLL codes of packing radius $2$ up to isometry.

To solve Problem \ref{prob:classification}, one 
may try to follow the approach in \cite{leung_lattice_2020}  to investigate the 
coefficients of the elements in $\hat{T}^{(4)}T_0$ and $\hat{T}^{(4)}T_1$ 
modulo $5$. In a very recent work \cite{zhou_moreAPLL_2022}, this approach has been applied and Problem \ref{prob:classification} has been solved except for only a finite number of $n$.

One may naturally expect  a similar classification result provided that the packing density equals $\frac{|S(n,2)|}{|S(n,2)|+\epsilon }$, where $\epsilon$ is a small integer compared with $n$. However, it seems not an easy task, because the order of the group $G$ equals $|S(n,2)|+\epsilon=2n^2+2n+1+\epsilon$ which is determined by $\epsilon$ and some key steps of the proof should rely on the structure of $G$. One may compare the proof for $\epsilon=0$ and  for $\epsilon =1$ with $n\equiv 0\pmod{3}$ in \cite[Proposition 3.3]{leung_lattice_2020} and Theorems \ref{th:n=0_odd}, \ref{th:n=0_even}. The first one consists of only a half-page proof and the second one is much more complicated. Therefore, a small change of the order of $G$ could lead to a big change of the proof. 

Another interesting question is about quasi-perfect linear Lee codes. As we have mentioned in the introduction, 
the packing density of the 2-quasi-perfect Lee codes constructed in \cite{camarero_quasi-perfect_lee_2016,mesnager_2-correcting_2018} tends to $\frac{1}{2}$ when $n\rightarrow \infty$.
However, to the best of our knowledge, there is no construction with packing density tending to $1$ when $n\rightarrow \infty$. Therefore, we propose the following open problem.
\begin{problem}
	Find better upper and lower bounds of the packing density of $2$-quasi-perfect linear Lee codes in $\Z^n$  for infinitely many $n$.
\end{problem}

\section*{Acknowledgment}
The authors express their gratitude to the anonymous reviewers for their detailed and constructive comments which are very helpful to the improvement of the presentation of this paper.
Xiaodong Xu is partially supported by the Natural Science Foundation of China (No.\ 71471176). 
Yue Zhou is partially supported by the Natural Science Foundation of Hunan Province (No.\ 2019RS2031) and the Training Program for Excellent Young  Innovators of Changsha (No.\ kq2106006).

\end{document}